\documentclass[12pt,reqno]{amsart}
\usepackage{times}
\usepackage{amsmath,amsfonts, amstext,amssymb,amsbsy,amsopn,amsthm}
\usepackage{bm}
\usepackage{dsfont}
\usepackage{esint}
\usepackage{graphicx}   
\usepackage{hyperref}
\usepackage[all]{xy}
\newcommand{\diam}{\text{diam}}

\newcommand{\dB}{\mathds{B}}
\newcommand{\dC}{\mathds{C}}
\newcommand{\dD}{\mathds{D}}
\newcommand{\dH}{\mathds{H}}

\newcommand{\dR}{\mathds{R}}

\newcommand{\dZ}{\mathds{Z}}

\DeclareMathOperator{\Conf}{Conf}

\DeclareMathOperator{\Div}{Div}
\DeclareMathOperator{\dvol}{dvol}

\DeclareMathOperator{\Int}{Int}
\DeclareMathOperator{\Isom}{Isom}
\DeclareMathOperator{\Ker}{Ker}

\DeclareMathOperator{\PV}{P.V.}

\DeclareMathOperator{\Res}{Res}
\DeclareMathOperator{\Ric}{Ric}

\DeclareMathOperator{\Supp}{supp}

\DeclareMathOperator{\Vol}{Vol}

\newtheorem{theorem}{Theorem}[section]

\newtheorem{assumption}[theorem]{Assumption}

\newtheorem{lemma}[theorem]{Lemma}
\newtheorem{proposition}[theorem]{Proposition}

\theoremstyle{definition}
\newtheorem{definition}[theorem]{Definition}
\theoremstyle{remark}
\newtheorem{remark}{Remark}[section]

\newtheorem{example}{Example}[section]

\theoremstyle{remark}

\numberwithin{equation}{section}

\begin{document}
\title{Nonlocal Curvature and Topology of Locally Conformally Flat Manifolds}
\date{\today}

\author{Ruobing Zhang}

\address{Department of Mathematics, Stony Brook University, Stony Brook, NY, 11790}

\email{rbzhang@math.stonybrook.edu}

\begin{abstract}
In this paper, we study the geometry of locally conformally flat manifolds $(M^n,g)$ with positive scalar curvature. Schoen-Yau proved that its universal cover $(\widetilde{M^n},\tilde{g})$ is conformally embedded in $\mathbb{S}^n$ such that $M^n$ is a Kleinian manifold. Moreover, the limit set of the Kleinian group has Hausdorff dimension less than $\frac{n-2}{2}$. If additionally we 
assume that the nonlocal curvature $Q_{2\gamma}\geq0$ for some $1<\gamma<2$, 
then we will prove that the Hausdorff dimension of the limit set is in fact at most $\frac{n-2\gamma}{2}$. This upper bound is sharp. As applications, we obtain some  
geometric rigidity and classification theorems for the manifolds with $Q_{2\gamma}\geq0$ in low dimensions. 
 \end{abstract}

\maketitle

\tableofcontents

\section{Introduction}\label{s:introduction}
Compact locally conformally flat manifolds with positive scalar curvature can be viewed as Kleinian manifolds by Schoen-Yau's fundamental work in \cite{SY}. That is, if $(M^n,g)$ is a compact locally conformally flat manifold with $R_g> 0$, then the universal cover $(\widetilde{M^n},\tilde{g})$
can be conformally embedded in the standard sphere $(\mathbb{S}^n,g_1)$.
Moreover, $\pi_1(M^n)$
is isomorphic to a Kleinian group $\Gamma\leq\Conf(\mathbb{S}^n)$ such that \begin{equation}\widetilde{M^n}\cong\Omega(\Gamma)\equiv \mathbb{S}^n\setminus \Lambda,
 \end{equation}
where $\Lambda\equiv\Lambda(\Gamma)$ is the limit set of the Kleinian group $\Gamma$.
In \cite{SY}, Schoen-Yau also proved the following Hausdorff dimension estimate on the limit set $\Lambda$ in the above setting,
\begin{equation}
\dim_{\mathcal{H}}(\Lambda)<\frac{n-2}{2}.
\end{equation}
The above Hausdorff dimension estimate immediately gives homotopy vanishing and homology vanishing results, which are interesting topological obstructions for conformally flat manifolds with nonnegative scalar curvature  (see \cite{SY} for more details).

In this paper, we will generalize the above theory to the fractional setting. In conformal geometry, scalar curvature $R_g$ arises as the zeroth order term of the conformal Laplacian operator. More precisely, denote $J_g\equiv\frac{R_g}{2(n-1)}$, then
\begin{equation}
P_{2}\equiv -\triangle_g+\frac{n-2}{2}J_g.
\end{equation}
It is standard that the second order  conformal Laplacian operator $P_{2}$ satisfies the following conformal covariance property: For $n\geq 3$, let $\hat{g}=v^{\frac{4}{n-2}}g$ and let $\widehat{P}_2$ be the conformal Laplacian with respect to the conformal metric $\hat{g}$, then
\begin{equation}
\widehat{P}_2(u)=v^{-\frac{n+2}{n-2}}P_2(uv).\label{conf-laplacian}
\end{equation}
The fourth order analogy of $P_2$ is called {\textit{Paneitz operator}}, which is defined by
\begin{equation}
P_4(u)\equiv(-\triangle_g)^2u+\Div_g(4A_g
\langle\nabla_g u, e_j\rangle e_j-(n-2)J_g\nabla_g u)+\frac{n-4}{2}Q_4\cdot u,
\end{equation}
 where 
 \begin{equation}
 A_g\equiv\frac{1}{n-2}\Big(\Ric_g-J_g\cdot g\Big).
 \end{equation}
 The above $Q_4=(\frac{n-4}{2})^{-1}P_4(1)$ is called Branson's $Q$ curvature.
Similar to (\ref{conf-laplacian}), Paneitz operator has the following 
conformal covariance property: For $n\geq 5$ and $\hat{g}=v^{\frac{4}{n-4}}g$, then
\begin{equation}
\widehat{P}_4(u)=v^{-\frac{n+4}{n-4}} P_4(uv).
\end{equation}
Chang-Hang-Yang studied the covering geometry similar to what Schoen-Yau did (see more details in \cite{CHY}). 

In our paper, we focus on the fractional order conformally covariant operator $P_{2\gamma}$ and the corresponding $Q_{2\gamma}$ curvature for $1<\gamma<2$. The fundamental analytic theory for $P_{2\gamma}$ is developed by Chang-Gonz\'alez in their work \cite{CG}. 
Throughout our paper, without explicitly stated, the operator $P_{2\gamma}(g,g_+)$ is always defined on {\it a  locally conformally flat manifold $(M^n,g)$ with a hyperbolic filling-in $(\dH^{n+1}/\Gamma,g_+)$ with} $\Gamma\equiv\pi_1(M^n)$ and $\sec_{g_+}\equiv-1$.
See Section \ref{s:preliminaries} for more details about that. The main part of this paper is to study the covering geometry of compact locally conformally flat manifold $(M^n,g)$ with $R_g>0$ and $Q_{2\gamma}\geq0$
for $1<\gamma<2$.

The main focus of our paper is the geometric effects of the nonlocal curvature $Q_{2\gamma}$.
Specifically, we will study the topology and geometry of local conformally flat manifolds with $Q_{2\gamma}\geq0$.
The main theorem is the following sharp estimate on the Hausdorff dimension of the limit set.

\begin{theorem}
\label{t:dim-estimate}
Let $(M^n,g)$ ($n\geq 3$) be a compact locally conformally flat manifold with positive scalar curvature. Let $\Gamma\equiv\pi_1(M^n)$ and let $\Lambda(\Gamma)$ be the limit set of $\Gamma$, then the following holds:
\begin{enumerate} \item If $Q_{2\gamma}>0$ for some $\gamma\in(1,\max\{2,n/2\})$, then
\begin{equation}
\dim_{\mathcal{H}}(\Lambda(\Gamma))<\frac{n-2\gamma}{2}.\label{hausdorff-dimension-bound}
\end{equation}
\item If  $Q_{2\gamma}\geq 0$ for some $\gamma\in(1,\max\{2,n/2\})$, then
\begin{equation}
\dim_{\mathcal{H}}(\Lambda(\Gamma))\leq\frac{n-2\gamma}{2}.\label{sharp-hausdorff-dimension-bound}
\end{equation}
\end{enumerate}
\end{theorem}

\begin{remark}
Example \ref{e:hyperbolic} shows that the Hausdorff dimension upper bound \eqref{sharp-hausdorff-dimension-bound} is sharp.
\end{remark}

Combined with a theorem in \cite{QR}, an immediate application of the above Hausdorff dimension estimate (Theorem \ref{t:dim-estimate}) is an existence result for  fractional order Yamabe problem. 

\begin{theorem}
Let $(M^n,g)$ be a compact locally conformally flat manifold with $R_g>0$. Assume that $Q_{2\gamma}>0$ for some $\gamma\in(1,\max\{2,\frac{n}{2}\})$, then for every $\gamma'\in(0,\gamma]$, there exists a smooth Riemannian metric $\hat{g}$ which is conformal to $g$ such that $\widehat{Q}_{2\gamma'}\equiv1$.
\end{theorem}

The above Hausdorff dimension estimate  in effect leads to the following topological consequences. From now on, we assume that $\Gamma\equiv\pi_1(M^n)$ is {\it torsion-free} (see the discussion at the end of Section \ref{s:preliminaries} for the necessity of this assumption). 

For $n=3$, we have the following rigidity theorem. The proof will be given in 
Section \ref{s:critical-case}.

\begin{theorem}
\label{t:3d-rigidity-theorem}

 Let  $(M^3,g)$ be a compact locally conformally flat manifold with positive scalar curvature.
  Assume that $Q_3\geq0$ on $M^3$, then the following properties hold:
  \begin{enumerate}
  \item If $Q_3(x)>0$ for some $x\in M^3$, then  $(M^3,g)$ is conformal to the round sphere $(\mathbb{S}^3,g_1)$ of constant curvature $1$.
    
  \item If $Q_3\equiv 0$, then $(M^3,g)$ is isometric to the Riemannian product space $(\mathbb{S}^1\times\mathbb{S}^2,dt^2\oplus g_r)$, where $g_r$ is a round metric on $\mathbb{S}^2$ of Gaussian curvature $\frac{1}{r^2}$ for some $r>0$.
\end{enumerate}

\end{theorem}

\begin{remark}
Theorem \ref{t:3d-rigidity-theorem} shows that Property (2) in Theorem \ref{t:dim-estimate} still holds when $n=3$ and $\gamma=\frac{3}{2}$, i.e.
$\dim_{\mathcal{H}}(\Lambda)=0
$. In fact, in this case, $\Gamma$ must be elementary and $\Lambda$ is empty or contains exactly two points.  

\end{remark}

\begin{remark}
The rigidity result in Theorem \ref{t:3d-rigidity-theorem} does not hold in the general context of conformally compact Einstein manifolds.  See 
Example \ref{e:ads-schwarzschild} for a counterexample.
\end{remark}

\begin{theorem}\label{t:euler-number-rigidity}
Let $(X^4,g_+)$ be a conformally compact Einstein manifold with a conformal infinity $(M^3,[g])$ which is locally conformally flat. Assume that $(M^3,g)$ has positive scalar curvature and $Q_3(g,g_+)\geq0$, then
\begin{equation}
\chi(X^4)\geq 0.\label{nonnegative-euler}
\end{equation}
Moreover, if $\chi(X^4)=0$, then the following holds:
\begin{enumerate}
\item $Q_3(g,g_+)\equiv0$. 

\item $(X^4,g_+)$ is isometric to $(\dR^3\times\mathbb{S}^1,g_{-1})$ with $\sec_{g_{-1}}\equiv-1$. 

\item  $(M^3,g)$ is isometric to the Riemannian product $(\mathbb{S}^1\times\mathbb{S}^2,dt^2\oplus g_r)$, where $g_r$ is a round metric on $\mathbb{S}^2$ of Gaussian curvature $\frac{1}{r^2}$ for some $r>0$.
\end{enumerate}

\end{theorem}

\begin{proof}[Proof of Theorem \ref{t:euler-number-rigidity}]
Inequality \eqref{nonnegative-euler} directly follows from Chern-Gauss-Bonnet formula \eqref{chern-gauss-bonnet}. Next, if $\chi(X^4)=0$, then \eqref{chern-gauss-bonnet} implies that $|W|_{g_+}\equiv0$ on $(X^4,g_+)$ and $Q_3(g,g_+)\equiv0$ on $(M^3,g)$.
Therefore, $(X^4,g_+)$ is hyperbolic and the corresponding rigidity follows from case (2) in Theorem \ref{t:3d-rigidity-theorem}.

\end{proof}

When $n=4$ and $5$, 
we have the following topological rigidity.
\begin{theorem}\label{t:topological-rigidity}
Given $n=4$ or $5$, let $(M^n,g)$ be a compact locally conformally flat manifold with positive scalar curvature, if one of the following holds:
\begin{enumerate}
\item $n=4$,\ $Q_{2\gamma}\geq 0$ for some $1<\gamma<2$;

\item $n=5$, $Q_{2\gamma}\geq 0$ for some $\frac{3}{2}<\gamma<2$;

\item $n=5$, $Q_{2\gamma}>0$ for some $\frac{3}{2}\leq\gamma<2$;

\end{enumerate}
then 
$M^n$ is diffeomorphic to $\mathbb{S}^n$ or $\#_{k}(\mathbb{S}^1\times\mathbb{S}^{n-1})$ for some $k\in\dZ_+$ .

\end{theorem}

\begin{remark}
In Theorem \ref{t:topological-rigidity}, if $\pi_1(M^n)$ is non-elementary, then $\Gamma\equiv\pi_1(M^n)$ is isomorphic to a Schottky group.
\end{remark}

\begin{remark}
In case (2), (3) of Theorem \ref{t:topological-rigidity}, the lower bounds are optimal. See Example \ref{e:hyperbolic}.  \end{remark}

\begin{proof}
[Proof of Theorem \ref{t:topological-rigidity}] 
First, if $\Gamma=\{e\}$, then $M^n$ is diffeomorphic to $\mathbb{S}^3$ given that $\Gamma$ is torsion free.  So we consider the case that $\Gamma\equiv\pi_1(M^n)\neq \{e\}$
is elementary. Since $R_g>0$, it is a classical result by \cite{GL} that $M^n$ is not covered by $\mathbb{T}^n$. Therefore $\Gamma$
has to be an infinite cyclic group generated by a loxodromic element, and hence $M^n$ is diffeomorphic to $\mathbb{S}^1\times \mathbb{S}^{n-1}$. So Theorem \ref{t:topological-rigidity}
holds in this case.

Now we switch to the general case in which $\Gamma\equiv\pi_1(M^n)$ is non-elementary. In each of the cases $(1)$,  $(2)$ and $(3)$,
by Theorem \ref{t:dim-estimate}, we have that $\dim_{\mathcal{H}}(\Lambda(\Gamma))<1$. Furthermore, 
Patterson-Sullivan's theorem (see theorem \ref{t:patterson-sullivan}) implies that $\delta(\Gamma)=\dim_{\mathcal{H}}(\Lambda(\Gamma))<1$.
Since $\Gamma$
is non-elementary, $0<\delta(\Gamma)<1$.
Applying theorem 6.1 in \cite{izeki}, $M^n$ is diffeomorphic to $\#_k(\mathbb{S}^1\times\mathbb{S}^{n-1})$ for $k\in\dZ_+$.
\end{proof}

\section{Preliminaries}
\label{s:preliminaries}

\subsection{Basics in Kleinian Groups}
We start with some brief review on the basic concepts about Kleinian groups and their useful properties 
in conformal geometry.

Let $\Gamma\leq\Conf(\mathbb{S}^n)\cong\Isom(\dB^{n+1})$ be a Kleinian group which by definition gives a properly discontinuous action on a non-empty subdomain $\Omega(\Gamma)\subset \mathbb{S}^n$.
The notion of limit set is fundamental in the study of Kleinian groups.
\begin{definition}[Limit set] \label{d:limit-set} Let $\Gamma$ be a  Kleinian group, then the limit set is defined by
\begin{equation}
\Lambda(\Gamma)\equiv\{x\in(\mathbb{S}^n,g_1)|\exists\gamma_j\in\Gamma, \ y\in(\mathbb{S}^n,g_1),\ \text{s.t.}\ \lim\limits_{j\to\infty}d_{g_1}(\gamma_j(y), x)\to0\}.
\end{equation}
\end{definition}
Immediately, by definition, $\Lambda(\Gamma)$ is a $\Gamma$-invariant closed subset in $\mathbb{S}^n$. Moreover, if $\Gamma$ is co-compact, then 
\begin{equation}
\Lambda(\Gamma)=\partial\Omega(\Gamma)=\mathbb{S}^n\setminus\Omega(\Gamma).
\end{equation}

If $\Gamma$ acts freely on $\Omega(\Gamma)$, then the quotient space $\Omega(\Gamma)/\Gamma$ is a locally conformall                                                                                                    y flat manifold. 
Schoen-Yau's result in \cite{SY} proved that there are a large class of locally conformally flat manifolds of the form $\Omega(\Gamma)/\Gamma$. More precisely, let $(M^n,g)$ be locally conformally flat, if $R_g\geq0$, then
the universal cover $\widetilde{M^n}$
can be conformally embedded in $\mathbb{S}^n$.  

Next we introduce the concept of Poincar\'e exponent and its applications in the study of conformal geometry.
\begin{definition}
[Poincar\'e exponent]\label{d:poincare-series-on-sphere} Let $\Gamma$ be a Kleinian group. Given $s>0$, define the Poincar\'e series as follows, 
\begin{equation}
\mathcal{P}_{c}(s;x)\equiv \sum\limits_{\gamma\in\Gamma}\|\gamma'(x)\|^s,
\end{equation}
where $\|\gamma'(x)\|$ is the length of the conformal factor. Poincar\'e exponent is defined by
\begin{equation}
\delta(\Gamma)\equiv\inf\Big\{s>0\Big|\mathcal{P}_{c}(s;x)<\infty,\ \forall x\in\mathbb{S}^n\Big\}.
\end{equation}
\end{definition}
There are several equivalent definitions of Poincar\'e exponent. Now we describe Poincar\'e exponent on a hyperbolic space $(\dB^{n+1}, g_{-1},0^{n+1})$
with $\sec_{g_{-1}}\equiv-1$.
\begin{lemma}
Let $\Gamma\leq\Isom(\dB^{n+1})\cong\Conf(\mathbb{S}^n)$ be a Kleinian group and denote by $\delta(\Gamma)$ the Poincar\'e exponent in terms of Definition \ref{d:poincare-series-on-sphere}, then the following holds:

\begin{enumerate}
\item On $(\dB^{n+1},g_{-1},0^{n+1})$, let 
\begin{equation}
\mathcal{P}_{e}(s;0^{n+1})\equiv\sum\limits_{\gamma\in\Gamma}e^{-sd_{g_{-1}}(0^{n+1},\gamma\cdot 0^{n+1})},
\end{equation}
then we have that\begin{equation}
\delta(\Gamma)=\inf\Big\{s>0\Big|\mathcal{P}_{e}(s;0^{n+1})<\infty\Big\}.
\end{equation}

\item Let $\mathfrak{n}(R)\equiv\#\{\gamma\in\Gamma|d_{g_{-1}}(\gamma\cdot 0^{n+1},0^{n+1})\leq R \}$, then
\begin{equation}
\delta(\Gamma)=\limsup\limits_{R\to\infty}\frac{\log\mathfrak{n}(R)}{R}.
\end{equation}

\end{enumerate}
\end{lemma}

\begin{example}\label{e:elementary-group}
Let $\Gamma$ be elementary and torsion-free, namely, $|\Lambda(\Gamma)|\leq 2$. Then $\Gamma$ must be an infinite cyclic group generated by a loxodromic element or a parabolic group of rank $k$ for some $k\geq 1$. In the above two cases, $\delta(\Gamma)=0$, $k/2$ respectively. 
\end{example}

The above concept of Poicnar\'e exponent gives many interesting consequences in conformal geometry. To this end, we start with some analytic preliminaries. Let $(M^n,g)$ be a closed locally conformally flat manifold and assume that $\widetilde{M^n}$ is conformally embedded in $\mathbb{S}^n$. Therefore, there is some subdomain $\Omega(\Gamma)\subset\mathbb{S}^n$ such that  
$M^n=\Omega(\Gamma)/\Gamma$
where $\Gamma\cong\pi_1(M^n)$
is a Kleinian group.

By stereographic projection, the universal cover $\widetilde{M^n}$
can be viewed as a subdomain in $\dR^n$ with the Riemannian covering metric $\tilde{g}=e^{2w}g_0$. 
Since $R_{g_0}\equiv0$, the conformal change $g_0=e^{-2w}\tilde{g}$ gives the following natural elliptic equation, 
\begin{equation}
-\triangle_{\tilde{g}}(e^{-\frac{n-2}{2}w})+\frac{n-2}{2}J_{\tilde{g}}e^{-\frac{n-2}{2}w}=0,
\end{equation}
where $J_{\tilde{g}}\equiv\frac{R_{\tilde{g}}}{2(n-1)}$.
Cheng-Yau's gradient estimate gives that for every $x\in\dR^n\setminus \Lambda$ with $B_{4R}(x)\subset\dR^n\setminus\Lambda$,
\begin{equation}
\sup\limits_{y\in B_R(x)}|\nabla_{\tilde{g}}\log e^{-\frac{n-2}{2}w(y)}|\leq C(n,\tilde{g},R).
\end{equation}
As a corollary, we have the following
Harnack inequality, there exists $C(n,\tilde{g},R)>0$ such that for every $y,z\in B_{R}(x)$
\begin{equation}
C^{-1}\cdot e^{w(z)}\leq  e^{w(y)}\leq C\cdot  e^{w(z)}.\label{cy-harnack}
 \end{equation}

We give some basic properties of the conformal factors, which will be used frequently in the paper.

\begin{lemma}[\cite{CQY}]\label{l:c_0-estimate-on-factor}
If $M^n\cong\Omega(\Gamma)/\Gamma$ is closed, then there exists $K>0$ such that for every $x\in\Omega(\Gamma)$,
\begin{equation}
K^{-1}\cdot d_0^{-1}(x,\Lambda)\leq e^{w(x)}\leq K\cdot d_0^{-1}(x,\Lambda),
\end{equation}
where $d_0$ is the distance function with respect to Euclidean metric.
\end{lemma}
Next we give an integrability lemma for the conformal factor $e^w$.
\begin{lemma}\label{l:integrability-lemma-of-conformal-factor} Let $\Gamma$ be a Kleinian group and  denote by $\delta_c\equiv\delta(\Gamma)\geq 0$, then  $e^w\in L_{loc}^{p}(\dR^n)$ for every $0<p< n-\delta_c$.

\end{lemma}

\begin{proof}

Let $\psi:\mathbb{S}^n\setminus \{N\}\rightarrow\dR^n$ be the stereographic projection.
First, we show that for every $0<p<n-\delta_c$, $e^{\tilde{w}}\equiv e^{w\circ\psi}\cdot\Big(\frac{1+|x|^2}{2}\Big)\in L^p(\mathbb{S}^n)$. In fact, for every $\gamma\in\Gamma$ and for every $\tilde{x}\in\mathbb{S}^n\setminus\Lambda$,
\begin{equation}
|\gamma'(\tilde{x})|\cdot e^{\tilde{w}(\gamma\cdot\tilde{x})}=e^{\tilde{w}(\tilde{x})}.
\end{equation}

Let $K\subset \Omega(\Gamma)$
be a bounded subset such that $F\subset K$. By (\ref{cy-harnack}), there exists some positive constant $C_0(\tilde{g},\diam_{\tilde{g}}(K),n)>0$ such that that for every $\tilde{x},\tilde{y}\in K$,
\begin{equation}
C_0^{-1}\cdot e^{\tilde{w}(\tilde{x})}\leq e^{\tilde{w}(\tilde{y})}\leq C_0\cdot e^{\tilde{w}(\tilde{x})}.\label{harnack}
\end{equation}

Now fix a fundamental domain $F\subset \Omega(\Gamma)$, then the integral over $\mathbb{S}^n$ can be reduced to the following sum,
\begin{eqnarray}
\int_{\Omega(\Gamma)}e^{p\cdot w(\tilde{x})}dv_{g_1}(\tilde{x})&=&\sum\limits_{\gamma\in\Gamma}\int_{\gamma\cdot F}e^{p\cdot w(\gamma\cdot\tilde{x})}dv_{g_1}(\gamma\cdot\tilde{x})\nonumber\\
&=&\sum\limits_{\gamma\in\Gamma}\int_{F}|\gamma'(\tilde{x})|^{n-p}\cdot e^{p\cdot w(\tilde{x})}dv_{g_1}(\tilde{x}).
\end{eqnarray}
Harnack inequality (\ref{harnack}) implies that for fixed $\tilde{x}_0\in F$, we have that for every $\tilde{x}\in F$,
\begin{equation}
e^{\tilde{w}(\tilde{x})}\leq C_0\cdot e^{\tilde{w}(\tilde{x}_0)}.
\end{equation}
Therefore, 
\begin{equation}
\int_{\Omega(\Gamma)}e^{p\cdot w(\tilde{x})}dv_{g_1}(\tilde{x})\leq C_0^p\cdot\Vol_{\tilde{g}}(F)\cdot \sum\limits_{\gamma\in\Gamma}|\gamma'(\tilde{x}_0)|^{n-p}.
\end{equation}
Since $p<n-\delta_c$, by the definition of Poincar\'e exponent, the above integral is finite. 

\begin{equation}
\int_E e^{p\cdot w(x)}dv_0(x)=\int_{\psi^{-1}(E)}e^{p\cdot\tilde{w}(\tilde{x})}\cdot\Big(\frac{1+|x|^2}{2}\Big)^{p+n}dv_1(\tilde{x})<\infty. 
\end{equation}

\end{proof}

Now we discuss the measure-theoretic properties of Kleinian groups.
\begin{definition}
[Hausdorff Content and Hausdorff Dimension]
Let $S\subset\dR^n$ and $d\in[0,+\infty)$, then the $d$-dimensional Hausdorff content of $S$ is  defined by
\begin{equation}
C_{\mathcal{H}}^d(S)\equiv\inf\Big\{\sum_i r_i^d:\text{ there is a cover of } S\text{ by balls with radii }r_i>0\Big\}.\end{equation}
The Hausdorff dimension of $S$ is defined by
\begin{equation}
\dim_{\mathcal{H}}(S)\equiv\inf\{d\geq0|C_{\mathcal{H}}^d(S)=0\}.
\end{equation}
\end{definition}

If $\Gamma$ is non-elementary,  Patterson-Sullivan characterized Poincar\'e exponent by the Hausdorff dimension of the limit set $\Lambda(\Gamma)$.
\begin{theorem}[Patterson-Sullivan]\label{t:patterson-sullivan}
Let $\Gamma$ be a non-elementary Kleinian group, then
\begin{equation}
\delta(\Gamma)=\dim_{\mathcal{H}}(\Lambda(\Gamma)).
\end{equation}

\end{theorem}

The following lemma is rather standard and it is the foundation of the Hausdorff dimension estimate in our paper. The proof is given in the paper \cite{CHY}. 
\begin{lemma}
\label{l:hausdorff-dim-estimate}
Let $K\subset\dR^n$ be compact. Denote by 
$d_0(x,K)\equiv\inf\{|x-y|:y\in K\}$ for $x\in\dR^n$. Assume that for some $R>0$ and $\alpha\geq 1$, we have $K\subset B_R(0^n)$ and
\begin{equation}
\int_{B_R(0^n)\setminus K}d_0(x,K)^{-\alpha}dx<\infty,
\end{equation}
then $\dim_{\mathcal{H}}(K)\leq n-\alpha$. In addition, if $\alpha\geq n$,
then $K=\emptyset$.
\end{lemma}
Notice that, with Kleinian group action, 
Lemma \ref{l:integrability-lemma-of-conformal-factor} is a converse of lemma \ref{l:hausdorff-dim-estimate}.

Now we end this section by stating a classical theorem by Schoen-Yau (\cite{SY}), which has been mentioned in Section \ref{s:introduction}.
\begin{theorem} 
Let $(M^n,g)$ be a closed locally conformally flat manifold with $R_g\geq 0$, then its universal $\widetilde{M^n}$ can be conformally embedded in $\mathbb{S}^n$. Furthermore,
$\dim_{\mathcal{H}}(\Lambda(\Gamma))\leq\frac{n-2}{2}
$,
where $\Gamma\cong\pi_1(M^n)$ and $\Lambda(\Gamma)$ is the limit set on $\mathbb{S}^n$.
 \end{theorem}
 \begin{remark}
 In the above theorem, if the assumption  $R_g\geq 0$ is replaced with $R_g>0$, then correspondingly we have $\dim_{\mathcal{H}}(\Lambda(\Gamma))<\frac{n-2}{2}$.
 \end{remark}

\subsection{Fractional GJMS Operator and Fractional Order Curvature}
\label{ss:fractiona-operator}

Applying scattering theory on 
conformally compact Einstein manifolds developed by Graham-Zworski (\cite{GZ}), 
Chang-Gonz\'alez defined the notation of fractional GJMS operator in their joint paper \cite{CG}. 
In this section, we briefly review the definition of fractional GJMS operator and the corresponding fractional $Q$ curvature.
We start with the basic notions for defining fractional GJMS operators. Let $(X^{n+1},M^n,g_+)$ be a conformally compact Einstein manifold with conformal infinity $(M^n,[h])$ which satisfies the following:

$(1)$ $M^n=\partial X^{n+1}$,

$(2)$ $g_+$ is  complete on $\Int(X^{n+1})$ such that $\Ric_{g_+}\equiv-ng_+$,

$(3)$ there exists a smooth defining function of $M^n$, denoted by $\rho$, such that $\rho^2g_+$ is compact on $\overline{X^{n+1}}$ with $\rho^2g_+|_{TM^n}\in[h]$
and
\begin{equation}
\begin{cases}
\rho(x)>0, \ x\in \Int(X^{n+1}),\\
\rho(x)=0, \ x\in M^n,\\
|\nabla\rho|(x)\neq0,\ x\in M^n.\\
\end{cases}
\end{equation}
The following fact is standard (see \cite{Graham} or \cite{Lee}).
\begin{lemma}[Geodesic defining function]\label{l:geodesic-def-funct}
Let $(M^n,[h])$ be the conformal infinity, then for every $h_0\in[h]$ there exists a unique defining function $r$ in a neighborhood of $M^n$ such that 

$(1)$ $r^2g_+|_{r=0}=h_0$,

$(2)$ there exists $\epsilon>0$
such that $|\nabla r|_{r^2g_+}\equiv 1$ on $M^n\times [0,\epsilon)$. 
\end{lemma}

Now we are in a position to define fractional GJMS operators (see \cite{CG} for more details). According to the above notations,   let $s\in\dC\setminus\{\frac{n}{2}\}$, we will consider the following Poisson equation
\begin{equation}
(\triangle_{g_+}+s(n-s))u=0
\end{equation}
and consider its generalized eigenfunction \begin{equation}
 u=Fy^{n-s}+Gy^s,\ F,G\in C^{\infty}(X^{n+1}),\label{solution-to-poisson}
 \end{equation}
 where $y$ is the geodesic defining function for a given boundary $(M^n,h_0)$
with $h_0\in[h]$ (see lemma \ref{l:geodesic-def-funct}). The scattering operator $S(s)$ is a Dirichlet-to-Neumann operator which is given by  the following: if $f\equiv F|_{y=0}$, then
\begin{equation}
S(s)f\equiv G|_{y=0}.\end{equation}
In \cite{GZ}, it was proved that $S(s)$
is a meromorphic family conformally covariant operators for $s$ with simple poles in $\frac{n}{2}+\dZ_+$.
Now given $\gamma\in(0,\frac{n}{2})$, the fractional GJMS operator $P_{2\gamma}$ is defined by
\begin{equation}
P_{2\gamma}f\equiv P_{2\gamma}[h_0,g_+]f\equiv 2^{2\gamma}\frac{\Gamma(\gamma)}{\Gamma(-\gamma)}S(\frac{n}{2}+\gamma)f.\label{regular-fractional-operator}
\end{equation}
Notice that the function $\Gamma(\gamma)$
cancels the simple poles of the scattering operator $S(\frac{n}{2}+\gamma)$, and thus the fractional GJMS operator $P_{2\gamma}$
is continuous when $\gamma\in(0,\frac{n}{2})$.
Moreover, if $\gamma\in(0,n/2)$, the operator $P_{2\gamma}$
satisfies the following conformal covariance property: let $\hat{h}=v^{\frac{4}{n-2\gamma}}h$ and let $\widehat{P}_{2\gamma}$ be the fractional GJMS operator with respect to $\hat{h}$, then 
\begin{equation}
\widehat{P}_{2\gamma}(u)=v^{-\frac{n+2\gamma}{n-2\gamma}}P_{2\gamma}(uv).
\end{equation}
for any $u\in C^{\infty}(M^n)$.
With the fractional GJMS operator $P_{2\gamma}$, we will define fractional $Q$ curvature. 
We start with a simpler case $0<\gamma<\frac{n}{2}$. Let $P_{2\gamma}$ be in (\ref{regular-fractional-operator}), then the corresponding fractional $Q$ curvature is defined by
\begin{equation}
Q_{2\gamma}\equiv(\frac{n-2\gamma}{2})^{-1}P_{2\gamma}(1).\label{regular-fractional-Q}
\end{equation}

In Section \ref{s:critical-case}, we will discuss the critical case where $n=3$ and $\gamma=\frac{3}{2}$.
Actually, in general,
we can modify (\ref{regular-fractional-operator})  and (\ref{regular-fractional-Q}) to define the operator $P_{2\gamma}$
and $Q_{2\gamma}$ curvature in the critical case $\gamma=\frac{n}{2}$. First, if $n$ is even and $\gamma=\frac{n}{2}\equiv k\in\dZ_+$, then $s=n$ is a simple of the scattering operator $S(s)$, while the function $\Gamma(-\gamma)$ has a simple pole at $\gamma=\frac{n}{2}$, so we define
\begin{equation}
P_{2k}\equiv c_k\cdot\Res_{s=n}S(s),
\end{equation}
where $c_k\equiv  (-1)^{k+1} 2^{2k}k!(k-1)!$.
Correspondingly, the curvature
$Q_n$ is given by \begin{equation}Q_n= c_k\cdot S(n)1.\end{equation}
Notice that, if $n$ is even, $S(s)1$ is holomorphic when $s=n$.
Now assume that $n$ is odd, then the operator $S(\frac{n}{2}+\gamma)$ is continuous at $\gamma=\frac{n}{2}$ and $\Gamma(-\gamma)$ is also continuous at $\gamma=\frac{n}{2}$. In this case, the definition of $P_{2\gamma}$
is identical to that in the case $\gamma<\frac{n}{2}$.
By \cite{GZ}, if $n$ is odd, $\lim\limits_{s\to n}S(s)1=0$, then the limit 
$\lim\limits_{\gamma\to n/2}(\frac{n-2\gamma}{2})^{-1}P_{2\gamma}(1)
$ exists and so we define
\begin{equation}
Q_{n}\equiv \lim\limits_{\gamma\to n/2}(\frac{n-2\gamma}{2})^{-1}P_{2\gamma}(1).
\end{equation}
By the computations in \cite{GZ}, when $\gamma=\frac{n}{2}$, the following conformal covariance property always holds ($n$ is either even or odd): if $\hat{g}=e^{2w}g$ then,
\begin{equation}
e^{nw}\widehat{Q}_n=Q_n+P_n(w).
\end{equation}
In the case of $n=3$, the curvature $Q_3$ is intimately connected to the topology of the underlying manifold.
Indeed, it is a combined result due to \cite{FG} and \cite{CQY2} that if $(X^4,g_+)$ is a conformally compact Einstein manifold with a conformal infinity $(M^3,[g])$, 
then
\begin{equation}
8\pi^2\chi(X^4)=\int_{X^4}|W|^2\dvol_{g_+}+2\int_{M^3}Q_3(g,g_+)\dvol_g.\label{chern-gauss-bonnet}
\end{equation}
The above identity can be viewed as Chern-Gauss-Bonnet 
formula in the context of conformally compact Einstein manifolds. This formula will be applied in the proof of geometric rigidity theorems. 

We end this section by giving a remark about the definition of the fractional GJMS operator. Given a conformally compact Einstein manifold $(X^{n+1}, g_+, M^n)$ with a conformal infinity $(M^n,h_0)$, then
by definition, the operator $P_{2\gamma}\equiv P_{2\gamma}[h_0,g_+]$ depends not only on the boundary geometry $(M^n,h_0)$ but also on 
the global geometry $(X^{n+1},g_+)$. In other words, given $(M^n,h)$, the definition of the operator $P_{2\gamma}$ on $M^n$ depends on the choice of the space 
$(X^{n+1},g_+)$.
In our paper, we focus on the compact manifolds $(M^n,g)$ which are locally conformally flat with $R_g>0$. Notice that, in this setting, $(M^n,g)$ is not covered by a torus and thus it is a Kleinian manifold with $\Gamma\cong\pi_1(M^n)\leq\Conf(\mathbb{S}^n)$ convex co-compact. Furthermore, if $\Gamma$ is torsion-free, then $(M^n,[g])$ can be viewed as the conformal infinity of the complete hyperbolic manifold $\mathbb{H}^{n+1}/\Gamma$ (otherwise, $\mathbb{H}^{n+1}/\Gamma$ is a hyperbolic orbifold with finite singularities).
Throughout this paper, we will assume that $\pi_1(M^n)$ is {\it torsion-free} which is necessary for the background of scattering theory from the above discussion, and without explicitly stated the fractional GJMS operator $P_{2\gamma}$ in this paper is always defined by the hyperbolic filling-in in the above context.

\section{Fractional Laplacians on Euclidean Space}

We reviewed in Section \ref{ss:fractiona-operator} the definition of GJMS operators on a conformally compact Einstein manifold. 
In this section, as a special case, we will focus on those operators on Euclidean space with hyperbolic filling-in which plays a crucial role in the proof of the main theorems of this paper.

\subsection{Harmonic Extension and Equivalent Definitions of Fractional Laplacians}

Let $\alpha\in(0,1)$, first we introduce the following space
\begin{equation}
L_{2\alpha}(\dR^n)\equiv\Big\{u:\dR^n\to\dR^1\Big|\int_{\dR^n}\frac{|u(x)|}{1+|x|^{n+2\alpha}}dx<\infty\Big\}
\end{equation}
which is a natural space in the definition of fractional Laplacian.
Now we briefly review 
Caffarelli-Silvestre's result about harmonic extensions in \cite{CS}.  
For fixed $\alpha\in(0,1)$, denote by $b\equiv 1-2\alpha\in(-1,1)$, the for every $f\in C^{\infty}(\dR^n)\cap L^{\infty}(\dR^n)$, there exists 
a unique function $U(x,y)\in L^{\infty}(\dR_+^{n+1})$
such that 
\begin{equation}
\begin{cases}
\Div(y^{b}\nabla U)(x,y)\equiv 0.\\
U(x,0)=f(x).
\end{cases}
\end{equation}
Moreover, $U(x,y)$
can be expressed as the following Poisson integral
\begin{equation}
U(x,y)\equiv C_{n,\alpha}\int_{\dR^n}\frac{y^{1-b}\cdot f(\xi)}{(|x-\xi|^2+y^2)^{\frac{n+1-b}{2}}}d\xi.
\end{equation}
In fact, the above constant $C_{n,\alpha}>0$ is chosen such that
\begin{equation}
C_{n,\alpha}\int_{\dR^n}\frac{y^{1-b}}{(|x-\xi|^2+y^2)^{\frac{n+1-b}{2}}}d\xi=1.\label{unit-integral}
\end{equation}
 
With the above harmonic extension, 
Caffarelli-Silvestre proved that the definition of fractional Laplacian given by harmonic extension (Definition \ref{d:def-by-extension}) is actually equivalent to the one defined in terms of Fourier transform (Definition \ref{d:def-by-fourier}) and equivalent to the one defined in terms of singular integral (Definition \ref{d:def-by-sing-int}) as well.

\begin{definition}\label{d:def-by-extension} Let $d_{n,\alpha}>0$ be some positive constant depending only on $n$ and $\alpha$, then we define 
\begin{equation}
(-\triangle)^{\alpha}f(x)\equiv d_{n,\alpha}\cdot\lim\limits_{y\to0+}-y^b\frac{\partial U}{\partial y}.
\end{equation}
\end{definition}

\begin{definition}\label{d:def-by-fourier} Let $f\in L_{2\alpha}(\dR^n)$ which is viewed as a tempered distribution, then the fractional Laplacian $(-\triangle)^{\alpha}$ is given by the following equation,
\begin{equation}
\mathcal{F}((-\triangle)^{\alpha}f)\equiv|\xi|^{2\alpha}\mathcal{F}(f).
\end{equation}

\end{definition}

\begin{definition}\label{d:def-by-sing-int}
Let $d_{n,\alpha}>0$ be some positive constant, then we define 
\begin{equation} 
(-\triangle)^{\alpha}f(x)\equiv d_{n,\alpha}\cdot \PV \int_{\dR^n}\frac{f(x)-f(y)}{|x-y|^{n+2\alpha}}dy,
\end{equation}
where $f\in L_{2\alpha}(\dR^n)$.
\end{definition}
In fact, the above definitions are equivalent to the following definition in terms of second order difference.
\begin{lemma}\label{d:difference}Let $0<\alpha<1$. If $f\in L_{2\alpha}(\dR^n)$, then \begin{equation}
(-\triangle)^{\alpha}f(x)=-d_{n,\alpha}\cdot\int_{\dR^n}\frac{f(x+y)+f(x-y)-2f(x)}{|y|^{n+2\alpha}}dy,
\end{equation}
for some positive constant $d_{n,\alpha}$.
\end{lemma}

The following lemma is straightforward.
\begin{lemma}\label{l:comparison}Given $f\in C^{\infty}(
\dR^n)\cap L^{\infty}(\dR^n)$, let $U_1(x,y)$ be the harmonic extension of $f^{\tau_1}$ for some $\tau_1> 0$ and let 
$U_2(x,y)$ be the harmonic extension of $f^{\tau_2}$ for some $\tau_2>\tau_1>0$. Then for every $(x,y)\in\dR_+^{n+1}$,   we have that 
\begin{equation}
(U_{1})^{\tau_2/\tau_1}(x,y)\leq U_2(x,y).
\end{equation}

\end{lemma}

\begin{proof}

Equivalently, we will prove that for every $(x,y)\in\dR_+^{n+1}$
\begin{equation}
U_1(x,y)\leq (U_2)^{\tau_1/\tau_2}(x,y).
\end{equation}
By Poisson's formula, we have the following,
\begin{equation}
U_1(x,y)=C_{n,\alpha}\int_{\dR^n }\frac{y^{1-b}\cdot f^{\tau_1}(\xi)}{(|x-\xi|^2+y^2)^{\frac{n+1-b}{2}}}d\xi,
\end{equation}
and
\begin{equation}
U_2(x,y)=C_{n,\alpha}\int_{\dR^n}\frac{ y^{1-b}\cdot f^{\tau_2}(\xi)}{(|x-\xi|^2+y^2)^{\frac{n+1-b}{2}}}d\xi.
\end{equation}
In fact,
\begin{eqnarray}
U_1(x,y)&=&C_{n,\alpha}\int_{\dR^n }\frac{y^{1-b}\cdot f^{\tau_1}(\xi)}{(|x-\xi|^2+y^2)^{\frac{n+1-b}{2}}}d\xi\nonumber\\&=&C_{n,\alpha}\Big(\int_{\dR^n}\Big(\frac{y^{1-b}}{(|x-\xi|^2+y^2)^{\frac{n+1-b}{2}}}\Big)^{1-\frac{\tau_1}{\tau_2}}\cdot\Big(\frac{y^{1-b}f^{\tau_2}(\xi)}{(|x-\xi|^2+y^2)^{\frac{n+1-b}{2}}}\Big)^{\frac{\tau_1}{\tau_2}}d\xi\Big)\nonumber\\
&\overset{\text{H\"older}}{\leq}&C_{n,\alpha}\Big(\int_{\dR^n}\frac{y^{1-b}}{(|x-\xi|^2+y^2)^{\frac{n+1-b}{2}}}d\xi\Big)^{1-\frac{\tau_1}{\tau_2}}\Big(\int_{\dR^n}\frac{y^{1-b}f^{\tau}(\xi)}{(|\xi|^2+y^2)^{\frac{n+1-b}{2}}}d\xi\Big)^{\frac{\tau_1}{\tau_2}}\nonumber\\
&=&\Big(C_{n,\alpha}\int_{\dR^n}\frac{y^{1-b}}{(|\xi|^2+y^2)^{\frac{n+1-b}{2}}}d\xi\Big)^{1-\frac{\tau_1}{\tau_2}}\Big(C_{n,\alpha}\int_{\dR^n}\frac{y^{1-b}f^{\tau}(\xi)}{(|x-\xi|^2+y^2)^{\frac{n+1-b}{2}}}d\xi\Big)^{\frac{\tau_1}{\tau_2}}\nonumber\\
&\overset{(\ref{unit-integral})}{=}&\Big(C_{n,\alpha}\int_{\dR^n}\frac{y^{1-b}f^{\tau}(\xi)}{(|x-\xi|^2+y^2)^{\frac{n+1-b}{2}}}d\xi\Big)^{\frac{\tau_1}{\tau_2}}\nonumber\\
&=&(U_2)^{\tau_1/\tau_2}(x,y).
\end{eqnarray}

\end{proof}

\subsection{Comparison Principle for Fractional Laplacians}

Graham and Zworski 
developed the deep connection between scattering theory 
and conformal geometry in their fundamental work \cite{GZ}.
The following proposition is a crucial observation in the proof of the main theorems. The proof follows from the combination of Lemma \ref{l:comparison} and the expansion (\ref{solution-to-poisson}). 
\begin{proposition}\label{p:harmonic-meets-scattering}Given $0<\alpha<1$ and $\tau>0$. Let $f\in C^{\infty}(\dR^n)\cap L^{\infty}(\dR^n)$ be positive, then for every $x\in\dR^n$, the quantity
\begin{equation}
\mathcal{L}_f(\tau)(x)\equiv\frac{\Big((-\triangle)^{\alpha}f^{\tau}\Big)(x)}{\tau\cdot f^{\tau}}
\end{equation}
is monotone decreasing in $\tau$.
\end{proposition}
\begin{proof}
It suffices to prove that for every $\tau_2>\tau_1>0$ and $x\in\dR^n$,
\begin{equation}
\frac{(-\triangle)^{\alpha}f^{\tau_1}}{\tau_1\cdot f^{\tau_1}}\geq \frac{(-\triangle)^{\alpha}f^{\tau_2}}{\tau_2\cdot f^{\tau_2}}.\end{equation}

Let $U_{1}$ be the harmonic extension of $f^{\tau_1}$
and let $U_2$ be the harmonic extension of $f^{\tau_2}$. On the other hand, let $s\equiv\frac{n+2\alpha}{2}$ and immediately
$2s-n\in(0,2)$.
Now
denote 
\begin{equation}
u_{1}\equiv y^{n-s}\cdot U_{1},\ u_{2}\equiv y^{n-s}\cdot U_{2},
\end{equation}
 and then by the calculations in \cite{CG}, we have that
 \begin{equation}
 \triangle_{\mathbb{H}^{n+1}}u_{1}+s(n-s)u_{1}=0,\  \triangle_{\mathbb{H}^{n+1}}u_{2}+s(n-s)u_{2}=0.
 \end{equation}
Hence, we obtain the Taylor expansion of $U_{1}$ in geodesic defining function 
near to the slice $\{y=0\}$:
\begin{equation}
U_{1}=\Big(f^{\tau_1}+O(y^2)\Big)+\Big((S(s)f^{\tau_1})\cdot y^{2s-n}+O(y^{2s-n+2})\Big),
\end{equation}
and 
\begin{equation}
U_{2}=\Big(f^{\tau_2}+O(y^2)\Big)+\Big((S(s)f^{\tau_2})\cdot y^{2s-n}+O(y^{2s-n+2})\Big).
\end{equation}
Notice that $2s-n\in(0,2)$, so near to the boundary $\{y=0\}$, it holds that 
\begin{eqnarray}
(U_1)^{\tau_2/\tau_1}&=&\Big(f^{\tau_1}+(S(s)f^{\tau_1})\cdot y^{2s-n}+O(y^2)\Big)^{\tau_2/\tau_1}\nonumber\\
&=&f^{\tau_2}\cdot\Big(1+\frac{\tau_2}{\tau_1}\cdot\frac{S(s)f^{\tau_1}}{f^{\tau_1}}\cdot y^{2s-n}+O(y^{2})\Big),\end{eqnarray}
and
\begin{equation}
U_2=f^{\tau_2}\cdot\Big(1+\frac{S(s)f^{\tau_2}}{f^{\tau_2}}\cdot y^{2s-n}+O(y^2)\Big).
\end{equation}

 By Lemma \ref{l:comparison}, it holds that $(U_1)^{\tau_2/\tau_1}\leq U_{2}$, which implies that
 \begin{equation}
 \frac{\tau_2}{\tau_1}\cdot\frac{S(s)f^{\tau_1}}{f^{\tau_1}}\leq \frac{S(s)f^{\tau_2}}{f^{\tau_2}}.
 \end{equation}
By definition, $(-\triangle)^{\alpha}\equiv 2^{2\alpha}\cdot\frac{\Gamma(\alpha)}{\Gamma(-\alpha)}\cdot S(s)$
with $s=\frac{n+2\alpha}{2}$, and thus 
\begin{equation}
\frac{(-\triangle)^{\alpha}f^{\tau_1}}{\tau_1\cdot f^{\tau_1}}\geq \frac{(-\triangle)^{\alpha}f^{\tau_2}}{\tau_2\cdot f^{\tau_2}}.\end{equation}
\end{proof}

\subsection{Smooth Approximation and Convergence Lemmas}

The Hausdorff dimension estimate in Theorem \ref{t:dim-estimate} relies on delicate estimates on the conformal factor $e^w$. A crucial point to see this is to apply  Proposition \ref{p:harmonic-meets-scattering} to the conformal factor $e^w$. 
However, Proposition \ref{p:harmonic-meets-scattering} works only for the globally smooth function in $\dR^n$, while
in our context, the function $e^w$ blows up along the limit set $\Lambda(\Gamma)$ (see Lemma \ref{l:c_0-estimate-on-factor}).
So it is necessary to approximate $e^w$ by a family of converging smooth functions and obtain the comparison for the smooth approximations. The key step in this procedure is to 
prove the comparison inequalities of the smoothing approximations indeed converge to the comparison inequality for $e^w$.
 So the primary goal of this section
is to define the smooth approximations and prove various convergence results.

Let us start with the definition of the standard mollifier.
Given $p\geq1 $, for every $f\in L_{loc}^p(\dR^n)$, define
\begin{equation}f_{\epsilon}(x)\equiv f*\eta_{\epsilon}(x)=\int_{\dR^n}f(x-y)\eta_{\epsilon}(y)dy,\label{def-mollifier}
\end{equation}
where 
\begin{equation}
\eta_{\epsilon}(x)\equiv\frac{1}{\epsilon^n}\eta(\frac{x}{\epsilon})\geq 0,\end{equation}
satisfies the following
\begin{equation}
\eta\in C_0^{\infty}(\dR^n), \ \Supp(\eta)\in B_1(0^n),\ \int_{\dR^n}\eta=1.
\end{equation}
Immediately, by the symmetry property of convolution, it holds that 
\begin{equation}
f_{\epsilon}(x)=\int_{\dR^n}f(x-y)\eta_{\epsilon}(y)dy=\int_{B_{\epsilon}(0^n)}f(x-y)\eta_{\epsilon}(y)dy.
\end{equation}

Now we present some basic properties of the standard mollifiers which will be frequently used in the later proofs of the paper.
\begin{lemma}\label{l:uniform-convergence-of-mollifier}Let $\Omega\subset \dR^n$ be an open subset. 
If $f\in C(\Omega)$, then for every compact subset $K\subset \Omega$, it holds that when $\epsilon\to0$,
\begin{equation}
\|f_{\epsilon}-f\|_{C^0(K)}\rightarrow0.
\end{equation}

\end{lemma}
The proof of Lemma \ref{l:uniform-convergence-of-mollifier} is straightforward and standard, so we omit the proof.

Next we introduce the convergence property of derivatives of mollifiers.
\begin{lemma}\label{l:derivative-converge}
Let $f\in C^2(B_{R}(p))$, then for every $x\in B_R(p)$, it holds that:

$(1)$ for every $1\leq j\leq n$,
\begin{equation}
\lim\limits_{\epsilon\to0}\frac{\partial f_{\epsilon}}{\partial x_j}(x)= \frac{\partial f}{\partial x_j }(x),
\end{equation}

$(2)$ for every $1\leq i\leq j\leq n$,
\begin{equation}
\lim\limits_{\epsilon\to0}\frac{\partial^2 f_{\epsilon}}{\partial x_j\partial x_i}(x)= \frac{\partial^2 f}{\partial x_j\partial x_i}(x).
\end{equation}

\end{lemma}

\begin{proof}
The proof of the lemma is standard. So we just give a quick outline of the proof of $(1)$. Let $e_j\equiv\frac{\partial}{\partial x_j}$, then by mean value theorem,
\begin{eqnarray}
\frac{f_{\epsilon}(x+he_j)-f_{\epsilon}(x)}{h}&=&\int_{B_{\epsilon}(0^n)}\frac{f(x+he_j-y)-f(x-y)}{h}\cdot\eta_{\epsilon}(y)dy\nonumber\\
&=&\int_{B_{\epsilon}(0^n)}\frac{\partial f}{\partial x_j}(x-y+\xi(h)e_j) \cdot \eta_{\epsilon}(y)dy,
\end{eqnarray}
 where $0\leq \xi(h)\leq h$. Applying dominated 
convergence theorem,
\begin{eqnarray}
\frac{\partial f_{\epsilon}}{\partial x_j}(x)&=&\lim\limits_{h\to0}\frac{f_{\epsilon}(x+he_j)-f_{\epsilon}(x)}{h}\nonumber\\
&=&\int_{B_{\epsilon}(0^n)}\frac{\partial f}{\partial x_j}(x-y) \cdot \eta_{\epsilon}(y)dy.
\end{eqnarray}
So it follows that 
\begin{equation}
\lim\limits_{\epsilon\to0}\frac{\partial f_{\epsilon}}{\partial x_j}(x)=\frac{\partial f}{\partial x_j}(x).
\end{equation}

\end{proof}

The following convergence property of mollifiers is useful in the proof of convergence properties of fractional Laplacians.
\begin{lemma}
\label{l:integral-converge-mollifiers}Given $\alpha\in(0,1)$ and assume that $f\in L_{2\alpha}(\dR^n)$, then the following property holds: let $\delta>0$, then for sufficiently small $\epsilon>0$, 
\begin{equation}
\int_{\dR^n\setminus B_{2\delta}(x_0)}\frac{|f_{\epsilon}|(y)}{|x_0-y|^{n+2\alpha}}dy\leq  
C_0(n)\int_{\dR^n\setminus B_{\delta}(x_0)}\frac{|f|(y)}{|x_0-y|^{n+2\alpha}}dy.\label{mollifier-int-bound}
 \end{equation}

\end{lemma}

\begin{proof}
Applying the definition of mollifiers and Fubini's theorem,
\begin{eqnarray}
\int_{\dR^n\setminus B_{2\delta}(x_0)}\frac{|f_{\epsilon}(y)|}{|x_0-y|^{n+2\alpha}}dy&=&\int_{\dR^n\setminus B_{2\delta}(x_0)}\Big|\int_{\dR^n}f(y-z)\eta_{\epsilon}(z)dz\Big|\cdot\frac{1}{|x_0-y|^{n+2\alpha}}dy\nonumber\\
&\leq &\int_{\dR^n\setminus B_{2\delta}(x_0)}\Big(\int_{B_{\epsilon}(0^n)}|f(y-z)|\eta_{\epsilon}(z)dz\Big)\cdot\frac{1}{|x_0-y|^{n+2\alpha}}dy\nonumber\\
&= &\int_{B_{\epsilon}(0^n)}\Big(\int_{\dR^n\setminus B_{2\delta}(x_0)}\frac{|f(y-z)|}{|x_0-y|^{n+2\alpha}}dy\Big)\eta_{\epsilon}(z)dz.\label{iterated-int}\end{eqnarray}
Since  $y\in \Omega\equiv\dR^n\setminus B_{2\delta}(x_0)$, $z\in B_{\epsilon}(0^n)$ and $\epsilon<<\delta$, so triangle inequality implies that
$\xi\equiv y-z\in\dR^n\setminus B_{\delta}(x_0)$ and thus
\begin{equation}
|x_0-y|=|x_0-\xi-z|\geq |x_0-\xi|-|z|\geq\frac{1}{2}|x_0-\xi|.
\end{equation}
Therefore,
\begin{equation}
\int_{\dR^n\setminus B_{2\delta}(x_0)}\frac{|f(y-z)|dy}{|x_0-y|^{n+2\alpha}}\leq C_0(n)\int_{\dR^n\setminus B_{\delta}(x_0)}\frac{|f(\xi)|d\xi}{|x_0-\xi|^{n+2\alpha}}.
\end{equation}
Plugging the above inequality into (\ref{iterated-int}), 
\begin{eqnarray}
\int_{\dR^n\setminus B_{2\delta}(x_0)}\frac{|f_{\epsilon}|(y)}{|x_0-y|^{n+2\alpha}}dy&\leq & C_0(n)\int_{B_{\epsilon}(0^n)}\Big(\int_{\dR^n\setminus B_{\delta}(x_0)}\frac{|f(\xi)|}{|x_0-\xi|^{n+2\alpha}}d\xi\Big)\eta_{\epsilon}(z)dz\nonumber\\
&= &C_0(n)\int_{\dR^n\setminus B_{\delta}(x_0)}\frac{|f(\xi)|}{|x_0-\xi|^{n+2\alpha}}d\xi.
\end{eqnarray}
So the proof is done.

\end{proof}

With the above basic properties of mollifiers, now we focus 
on the convergence properties of fractional Laplacian $(-\triangle)^{\alpha}$ acting on the standard mollifiers $f_{\epsilon}$.
In the following proofs, mainly we will choose the definition of $(-\triangle)^{\alpha}$ in terms of second order difference (see Definition \ref{d:difference}).

Let $\Lambda\subset\dR^n$ be  a closed  subset of zero measure. 
From now on, we will study the  functions which satisfy the following assumption:

\begin{assumption} \label{a:functional-assumption} $f\in L_{2\alpha}(\dR^n)\cap C^{\infty}(\dR^n\setminus\Lambda)$ such that there exists $K>0$, $q>0$ such that for every $y\in\dR^n\setminus\Lambda$,
\begin{equation}K^{-1}\cdot d_0^{-q}(y,\Lambda)\leq f(y)\leq K\cdot d_0^{-q}(y,\Lambda),\end{equation}
 where $d_0$
is the Euclidean distance on $\dR^n$.
\end{assumption}

\begin{lemma}\label{l:basic-convergence-lemma}Fix $0<\alpha<1$ and $\Lambda\subset\dR^n$ be a closed subset of zero measure. Let $f\in L_{2\alpha}(\dR^n)\cap C^{\infty}(\dR^n\setminus\Lambda)$ and denote by $f_{\epsilon}$  the standard mollifier of 
$f$ in the sense of (\ref{def-mollifier}).
 Then \begin{equation}
\lim\limits_{\epsilon\to0}((-\triangle)^{\alpha}f_{\epsilon})(x)=((-\triangle)^{\alpha}f)(x).
\end{equation}
for every $x\in\dR^n\setminus\Lambda$.\end{lemma}

\begin{proof}

First, by Definition \ref{d:difference},
\begin{equation}
(-\triangle)^{\alpha}f_{\epsilon}(x)=-C_{n,\alpha}\int_{\dR^n}\frac{f_{\epsilon}(x+y)+f_{\epsilon}(x-y)-2f_{\epsilon}(x)}{|y|^{n+2\alpha}}dy.
\end{equation}
Let \begin{equation}H_{\epsilon}(x,y)\equiv \frac{f_{\epsilon}(x+y)+f_{\epsilon}(x-y)-2f_{\epsilon}(x)}{|y|^{n+2\alpha}},\end{equation}
and
\begin{equation}
H(x,y)\equiv \frac{f(x+y)+f(x-y)-2f(x)}{|y|^{n+2\alpha}}.
\end{equation}
Then for fixed $0<\delta_0(x)<\frac{1}{10^3}\cdot d_0(x,\Lambda)$,
\begin{eqnarray}
(-\triangle)^{\alpha}f_{\epsilon}(x)&=&-C_{n,\alpha}\int_{B_{\delta_0}(0^n)}H_{\epsilon}(x,y)dy-C_{n,\alpha}\int_{\dR^n\setminus B_{\delta_0}(0^n)}H_{\epsilon}(x,y)dy\nonumber\\
&\equiv &I_{\epsilon}^{(1)}(x)+I_{\epsilon}^{(2)}(x).\end{eqnarray}
Correspondingly,
\begin{equation}
I^{(1)}(x)\equiv -C_{n,\alpha}\int_{B_{\delta_0}(0^n)}\frac{f(x+y)+f(x-y)-2f(x)}{|y|^{n+2\alpha}}dy,
\end{equation}
and
\begin{equation}
I^{(2)}(x)\equiv -C_{n,\alpha}\int_{\dR^n\setminus B_{\delta_0}(0^n)}\frac{f(x+y)+f(x-y)-2f(x)}{|y|^{n+2\alpha}}dy.
\end{equation}
So we will prove that for every $x\in\dR^n\setminus\Lambda$, 
\begin{equation}
\lim\limits_{\epsilon\to0}I_{\epsilon}^{(1)}(x)=I^{(1)}(x),\label{I1-converge}
\end{equation}
and
\begin{equation}
\lim\limits_{\epsilon\to0}I_{\epsilon}^{(2)}(x)=I^{(2)}(x).\label{I2-converge}
\end{equation}

The first stage is to prove (\ref{I1-converge}). The proof  follows from  the derivatives convergence property in Lemma \ref{l:derivative-converge}.
In fact,
\begin{eqnarray}
|I_{\epsilon}^{(1)}(x)-I^{(1)}(x)|=-C_{n,\alpha}
\int_{B_{\delta_0}(0^n)}(H_{\epsilon}(x,y)-H(x,y))dy,
\end{eqnarray}
and
\begin{eqnarray}
|H_{\epsilon}(x,y)-H(x,y)|&=&\frac{|(f_{\epsilon}-f)(x+y)+(f_{\epsilon}-f)(x-y)-2(f_{\epsilon}-f)(x)|}{|y|^{n+2\alpha}}\nonumber\\
&\leq& C(n)\frac{\|\nabla^2(f_{\epsilon}-f)\|_{L^{\infty}(B_{4\delta_0}(x))}}{|y|^{n+2\alpha-2}}\nonumber\\
&\leq &C(n)\frac{\|\nabla^2 f\|_{L^{\infty}(B_{4\delta_0}(x))}}{|y|^{n+2\alpha-2}}.
\end{eqnarray}
Notice that the function $|y|^{n+2\alpha-2}\in L^1(B_{\delta_0}(0^n))$. Hence, applying Lemma \ref{l:uniform-convergence-of-mollifier} and 
dominated convergence theorem,
the convergence 
\begin{equation}
\lim\limits_{\epsilon\to0}I_{\epsilon}^{(1)}(x)=I^{(1)}(x)
\end{equation}
holds for every $x\in\dR^n\setminus\Lambda$.

Next we switch to prove (\ref{I2-converge}). The proof basically follows from Lemma \ref{l:integral-converge-mollifiers}. In fact,
\begin{eqnarray}
|I_{\epsilon}^{(2)}(x)-I^{(2)}(x)|&=&C_{n,\alpha}\int_{\dR^n\setminus B_{\delta_0}(0^n)}|H_{\epsilon}(x,y)-H(x,y)|dy\nonumber\\
&\leq &C_{n,\alpha}\Big(\int_{\dR^n\setminus B_{\delta_0}(0^n)}\frac{|(f_{\epsilon}-f)(x+y)|}{|y|^{n+2\alpha}}dy+
\int_{\dR^n\setminus B_{\delta_0}(0^n)}\frac{|(f_{\epsilon}-f)(x-y)|}{|y|^{n+2\alpha}}dy\nonumber\\
&&+2\int_{\dR^n\setminus B_{\delta_0}(0^n)}\frac{|f_{\epsilon}(x)-f(x)|}{|y|^{n+2\alpha}}dy\Big)\nonumber\\
&\equiv&C_{n,\alpha} (D_{\epsilon}(x)+F_{\epsilon}(x)+2G_{\epsilon}(x)).
\end{eqnarray}
Notice that, for every $x\in\dR^n\setminus\Lambda$,
\begin{equation}
D_{\epsilon}(x)=F_{\epsilon}(x)=\int_{\dR^n\setminus B_{\delta_0}(x)}\frac{|(f_{\epsilon}-f)(y)|}{|x-y|^{n+2\alpha}}dy,
\end{equation}
and thus by Lemma \ref{l:integral-converge-mollifiers} and dominated convergence theorem,
\begin{equation}
\lim\limits_{\epsilon\to0}D_{\epsilon}(x)=\lim\limits_{\epsilon\to0}F_{\epsilon}(x)=0.
\end{equation}
On the other hand, by Lemma \ref{l:uniform-convergence-of-mollifier}, for every $x\in\dR^n\setminus \Lambda$, \begin{equation}\lim\limits_{\epsilon\to0}G_{\epsilon}(x)=0.
\end{equation}
Now the proof of (\ref{I2-converge}) is done, and hence the proof of the lemma is complete.

\end{proof}

\begin{lemma}\label{l:main-convergence-lemma}Given $0<\alpha<1$ and $\tau>0$, let $f$ satisfy Assumption \ref{a:functional-assumption} and $f\in L_{loc}^{\tau}(\dR^n)$. For every fixed $x\in\dR^n\setminus\Lambda$,  choose $\delta_0(x)\equiv\frac{1}{10^6}\cdot\min\{d_0(x,\Lambda),\delta_1\}$ such that $f(y)>10^3$ for every $y\in T_{10\delta_1}(\Lambda)\equiv\{y\in\dR^n|d_0(y,\Lambda)\leq 10\delta_1\}$. Then the following holds for every $x\in\dR^n\setminus\Lambda$:

\begin{enumerate}

\item 
\begin{equation}\lim\limits_{\epsilon\to0}
\int_{\dR^n\setminus (B_{\delta_0}(x)\cup T_{\delta_0}(\Lambda))}\frac{|(f_{\epsilon})^{\tau}(y)-f^{\tau}(y)|}{|x-y|^{n+2\alpha}}dy=0.
\end{equation}

\item 
\begin{equation}
\lim\limits_{\epsilon\to0}\int_{B_{\delta_0}(0^n)}\frac{|((f_{\epsilon})^{\tau}-f^{\tau})(x+y)+((f_{\epsilon})^{\tau}-f^{\tau})(x-y)-2((f_{\epsilon})^{\tau}-f^{\tau})(x)|}{|y|^{n+2\alpha}}dy=0.
\end{equation}

\item 
\begin{equation}
\lim\limits_{\epsilon\to0}\int_{T_{\delta_0}(\Lambda)}\frac{|(f_{\epsilon})^{\tau}(y)-f^{\tau}(y)|}{|x-y|^{n+2\alpha}}dy= 0.
\end{equation}

\item \begin{equation}
\lim\limits_{\epsilon\to0}\Big((-\triangle)^{\alpha}(f_{\epsilon})^{\tau}\Big)(x)= \Big((-\triangle)^{\alpha}f^{\tau}\Big)(x).
 \end{equation}

\item  

\begin{equation}
\lim\limits_{\epsilon\to0}\Big((-\triangle)^{\alpha}\log (f_{\epsilon})\Big)(x)= \Big((-\triangle)^{\alpha}\log f\Big)(x).
\end{equation}

\end{enumerate}

\end{lemma}

\begin{proof}
$(1)$ follows from $C^0$-estimate of $f$.  By Assumption \ref{a:functional-assumption},
there exists $K_0>0$, $p>0$ such that for every $y\in\dR^n\setminus T_{\delta_0/2}(\Lambda)$, we have that
\begin{equation}
\frac{1}{K_0}\cdot (d_0(y,\Lambda))^{-\tau\cdot q }\leq f^{\tau}(y)\leq K_0\cdot (d_0(y,\Lambda))^{-\tau\cdot q}, 
\end{equation}
and thus for every $y\in\dR^n\setminus T_{\delta_0/2}(\Lambda)$
\begin{equation}
|f^{\tau}(y)|\leq 2^{\tau\cdot q}\cdot K_0\cdot\delta_0^{-\tau\cdot q}\label{c_0-estimate-away-from-limit-set}. 
\end{equation}
By the definition of $f_{\epsilon}$, it holds that for every $y\in\dR^n\setminus T_{\delta_0}(\Lambda)$,
\begin{eqnarray}
|(f_{\epsilon})^{\tau}(y)|&=&\Big|\Big(\int_{\dR^n}f(y-z)\cdot\eta_{\epsilon}(z)dz\Big)^{\tau}\Big|\nonumber\\
&=&\Big|\Big(\int_{B_{\epsilon}(0^n)}f(y-z)\cdot\eta_{\epsilon}(z)dz\Big)^{\tau}\Big|\nonumber\\
&\leq &\|f\|_{L^{\infty}(B_{\epsilon}(y))}\cdot\Big|\int_{B_{\epsilon}(0^n)}\eta_{\epsilon}(z)dz\Big|^{\tau}\nonumber\\
&=&\|f\|_{L^{\infty}(B_{\epsilon}(y))}.
\end{eqnarray}
Inequality (\ref{c_0-estimate-away-from-limit-set}) gives that for every $y\in\dR^n\setminus T_{\delta_0}(\Lambda)$, 
\begin{equation}
|(f_{\epsilon})^{\tau}(y)|\leq   2^{\tau\cdot q}\cdot K_0\cdot\delta_0^{-\tau\cdot q}.
\end{equation}

Consequently,
\begin{eqnarray}
\int_{\dR^n\setminus (B_{\delta_0}(x)\cup T_{\delta_0}(\Lambda))}\frac{|(f_{\epsilon})^{\tau}(y)-f^{\tau}(y)|}{|x-y|^{n+2\alpha}}dy&\leq&\frac{2^{\tau\cdot q+1}\cdot K_0\cdot \delta_0^{-\tau\cdot q}}{\delta_0^{\tau\cdot q}}\int_{\dR^n\setminus B_{\delta_0}(x)}\frac{dy}{|x-y|^{n+2\alpha}}\nonumber\\
&\leq& \frac{2^{\tau\cdot q}\cdot K_0\cdot\delta_0^{-\tau\cdot q-2\alpha}}{\alpha},
\end{eqnarray}
and hence dominated convergence theorem implies that for every fixed $x\in \dR^n\setminus\Lambda$,
\begin{equation}
\int_{\dR^n\setminus (B_{\delta_0}(x)\cup T_{\delta_0}(\Lambda))}\frac{|(f_{\epsilon})^{\tau}(y)-f^{\tau}(y)|}{|x-y|^{n+2\alpha}}dy\rightarrow0.\end{equation}

$(2)$ can be immediately proved by the $C^2$-derivative estimate.
\begin{eqnarray}
&&\int_{B_{\delta_0}(0^n)}\frac{|((f_{\epsilon})^{\tau}-f^{\tau})(x+y)+((f_{\epsilon})^{\tau}-f^{\tau})(x-y)-2((f_{\epsilon})^{\tau}-f^{\tau})(x)|}{|y|^{n+2\alpha}}dy\nonumber\\
&\leq &\int_{B_{\delta_0}(0^n)}\frac{\|\nabla^2((f_{\epsilon})^{\tau}-f^{\tau})\|_{L^{\infty}(B_{2\delta_0}(x))}}{|y|^{n-2+2\alpha}}dy.\label{alpha-small-ball}
\end{eqnarray}
By straightforward calculations,
\begin{equation}
\nabla^2((f_{\epsilon})^{\tau}-f^{\tau})=\tau\Big((f_{\epsilon})^{\tau-1}(\nabla^2f_{\epsilon})-f^{\tau-1}(\nabla^2f) \Big)+\tau(\tau-1)\Big(f_{\epsilon}^{\tau-2}df_{\epsilon}\otimes df_{\epsilon}-f^{\tau-2}df\otimes df\Big).
\end{equation}

The pointwise convergence property of $\nabla f_{\epsilon}$ and $\nabla^2 f_{\epsilon}$ implies that 
the integral (\ref{alpha-small-ball}) is finite.  Therefore, dominated convergence theorem gives the desired convergence.

$(3)$ follows from local integrability property, mean value theorem and H\"older's inequality. From now on, we assume that $\tau>0$ and $\tau\neq 1$ because the case $\tau=1$ has been proven in Lemma \ref{l:basic-convergence-lemma}. 
Let 
\begin{equation}
W(y)\equiv\tau\cdot \max\{|f_{\epsilon}|^{\tau-1}(y),|f|^{\tau-1}(y)\},
\end{equation}
and then by Lagrange mean value theorem, it holds that
\begin{equation}
\int_{T_{\delta_0}(\Lambda)}\frac{|(f_{\epsilon})^{\tau}(y)-f^{\tau}(y)|}{|x-y|^{n+2\alpha}}dy\leq \int_{T_{\delta_0}(\Lambda)}\frac{W(y)\cdot |f_{\epsilon}-f|(y)}{|x-y|^{n+2\alpha}}dy.
\end{equation}
Since $\delta_0(x)<\frac{d_0(x,\Lambda)}{10^6}$ and $d_0(y,\Lambda)<\delta_0$, so triangle inequality implies that\begin{equation}
|x-y|=d_0(x,y)>>\delta_0.
\end{equation}
Therefore,
\begin{equation}\int_{T_{\delta_0}(\Lambda)}\frac{|(f_{\epsilon})^{\tau}(y)-f^{\tau}(y)|}{|x-y|^{n+2\alpha}}dy\leq  \frac{1}{\delta_0^{n+2\alpha}}\int_{T_{\delta_0}(\Lambda)}W(y)\cdot |f_{\epsilon}-f|(y)dy.\label{factorize}\end{equation}

$(i)$ When $0<\tau< 1$, by definition, $|W(y)|<1$
 for every $y\in T_{\delta_0}(\Lambda)$, which implies that 
 \begin{equation}\int_{T_{\delta_0}(\Lambda)}\frac{|(f_{\epsilon})^{\tau}(y)-f^{\tau}(y)|}{|x-y|^{n+2\alpha}}dy\leq\frac{1}{\delta_0^{n+2\alpha}} \int_{T_{\delta_0}(\Lambda)}|f_{\epsilon}-f|(y)dy\to0.
 \end{equation}
 
$(ii)$  Next, we consider the case $\tau>1$.
Let $p\equiv\tau$ and $q=\frac{\tau}{\tau-1}$, then $\frac{1}{p}+\frac{1}{q}=1$. Now applying H\"older inequality to the right hand side of inequality (\ref{factorize}), we have that
\begin{eqnarray}
\int_{T_{\delta_0}(\Lambda)}\frac{|(f_{\epsilon})^{\tau}(y)-f^{\tau}(y)|}{|x-y|^{n+2\alpha}}dy\leq \frac{1}{\delta_0^{n+2\alpha}}\Big(\int_{T_{\delta_0}(\Lambda)}|W(y)|^qdy\Big)^{1/q}\cdot \Big(\int_{T_{\delta_0}(\Lambda)}|f_{\epsilon}-f|^p\Big)^{1/p}.\label{pq-integral}
\end{eqnarray}
Now we estimate the uniform bound (independent of $\epsilon$) of the integral $\int_{T_{\delta_0}(\Lambda)}|W(y)|^qdy$.
In fact, \begin{eqnarray}
\Big(\int_{T_{\delta_0}(\Lambda)}|f_{\epsilon}|^{(\tau-1)q}\Big)^{1/q}&=&\Big(\int_{T_{\delta_0}(\Lambda)}|f_{\epsilon}|^{\tau}\Big)^{(\tau-1)/\tau}\nonumber\\
&=&\Big(\int_{T_{\delta_0}(\Lambda)}|(f_{\epsilon}-f)+f|^{\tau}\Big)^{(\tau-1)/\tau}\nonumber\\
&\leq &\Big(\Big(\int_{T_{\delta_0}(\Lambda)}|f_{\epsilon}-f|^{\tau}\Big)^{1/\tau}+\Big(\int_{T_{\delta_0}(\Lambda)}|f|^{\tau}\Big)^{1/\tau}\Big)^{\tau-1}.\label{mollifier-tau} 
\end{eqnarray}
For $\epsilon\to0$, it holds that
\begin{equation}
\int_{T_{\delta_0}(\Lambda)}|f_{\epsilon}-f|^{\tau}\rightarrow0,\end{equation}
and then (\ref{mollifier-tau}) implies that 
\begin{equation}
\Big(\int_{T_{\delta_0}(\Lambda)}|f_{\epsilon}|^{(\tau-1)q}\Big)^{1/q}\leq 2^{\tau-1}\cdot\Big(\int_{T_{\delta_0}(\Lambda)}|f|^{\tau}\Big)^{(\tau-1)/\tau}.
\end{equation}
Immediately by the definition of the function $W$,  we have that \begin{equation}
\Big(\int_{T_{\delta_0}(\Lambda)}|W(y)|^{(\tau-1)q}dy\Big)^{1/q}\leq 2^{\tau-1}\cdot\Big(\int_{T_{\delta_0}(\Lambda)}|f(y)|^{\tau}dy\Big)^{(\tau-1)/\tau}.
\end{equation}

Therefore, when $\epsilon\to0$, the above computations and
(\ref{pq-integral}) give that 
\begin{equation}
\int_{T_{\delta_0}(\Lambda)}\frac{|(f_{\epsilon})^{\tau}(y)-f^{\tau}(y)|}{|x-y|^{n+2\alpha}}dy\rightarrow0.
\end{equation}
The proof of $(3)$ is done.

$(4)$ easily follows from the combination of $(1)$, $(2)$
and $(3)$.

\vspace{0.3cm}

Now we proceed to prove $(5)$. First notice that, if $f$ satisfies Assumption \ref{a:functional-assumption} and $f\in L_{loc}^{\tau}(\dR^n)$ for some $\tau>1$, then it holds that $\log f\in L_{2\alpha}(\dR^n)$ and $\log (f_{\epsilon})\in L_{2\alpha}(\dR^n)$ which means that both $(-\triangle)^{\alpha}\log f$ and $(-\triangle)^{\alpha}\log(f_{\epsilon})$ are well defined. 
To see this, for sufficiently small $\delta_0>0$, we have that 
$1<<\log f< f$ on $T_{\delta_0}(\Lambda)$ and thus $\log f\in L_{loc}^{\tau}(\dR^n)$. In particular, $\log f\in L_{loc}^{\tau}(\dR^n)$. Moreover,
by Assumption \ref{a:functional-assumption},
\begin{equation}
K^{-1}\cdot d_0^{-q}(y,\Lambda)\leq f(y)\leq K\cdot d_0^{-q}(y,\Lambda),\end{equation}
 and thus for sufficiently large $R>>1$, it holds that
 \begin{eqnarray}\int_{\dR^n\setminus B_R(0^n)}\frac{\log f(y)}{1+|y|^{n+2\alpha}}dy&\leq &
 \int_{\dR^n\setminus B_R(0^n)}\frac{C\cdot \log |y|}{|y|^{n+2\alpha}}dy\nonumber\\
 &\leq & \int_{\dR^n\setminus B_R(0^n)}\frac{C |y|^{m}}{|y|^{n+2\alpha}}dy<\infty,
 \end{eqnarray}
where the exponent $m>0$ can be chosen such that $0<m<\alpha$. Therefore, $\log f\in L_{2\alpha}(\dR^n)$. Similarly, one can show that $\log (f_{\epsilon})\in L_{2\alpha}(\dR^n)$.

With the above preliminary result, we will show the convergence property. The above arguments actually shows that for every $x\in\dR^n\setminus\Lambda$,
\begin{equation}
\lim\limits_{\epsilon\to0}
\int_{\dR^n\setminus (B_{\delta_0}(x)\cup T_{\delta_0}(\Lambda))}\frac{|\log(f_{\epsilon})(y)-\log f(y)|}{|x-y|^{n+2\alpha}}dy=0.
\end{equation}
Moreover, the convergence of the integral in $B_{\delta_0}(0^n)$ is the same as that in (2).
 Now it suffices to prove that for every $x\in\dR^n\setminus\Lambda$,
\begin{equation}
\lim\limits_{\epsilon\to0}\int_{T_{\delta_0}(\Lambda)}\frac{|(\log (f_{\epsilon}))(y)-\log f(y)|}{|x-y|^{n+2\alpha}}=0.
\end{equation}
The basic idea is the same as that in $(3)$, for the completion we still give the detailed proof here. We choose $\delta_0(x)<\frac{1}{10^6}d_0(x,\Lambda)$
such that $ f (y)>10^3$ for every $y\in T_{10\delta_0}(\Lambda)$.

It is by definition that $f(y)>0$ and $f_{\epsilon}(y)>0$. We define
\begin{equation}
W(y)\equiv\max\{(f_{\epsilon})^{-1}(y), f^{-1}(y)\},\ y\in\dR^n,\label{bound-of-W}
\end{equation}
then $0<W(y)<1$ for every $y\in\dR^n$.
So Lagrangian mean value theorem and \eqref{bound-of-W} imply that
\begin{eqnarray}
\int_{T_{\delta_0}(\Lambda)}\frac{|(\log (f_{\epsilon}))(y)-\log f(y)|}{|x-y|^{n+2\alpha}}&\leq& \delta_0^{-n-2\alpha}\int_{T_{\delta_0}(\Lambda)}|W(y)|\cdot|f_{\epsilon}(y)-f(y)|dy\nonumber\\
&\leq& \delta_0^{-n-2\alpha}\int_{T_{\delta_0}(\Lambda)}|f_{\epsilon}(y)-f(y)|dy\nonumber\\
&\overset{\epsilon\to0}{\leq}&0.
\end{eqnarray}
The proof is done.

\end{proof}

\begin{example}
 Let $(M^n,g)$ be a Kleinian manifold, then its universal cover $(\widetilde{M^n},\tilde{g})$ satisfies $\tilde{g}=e^{2w}g_0$ where $g_0$ is the Euclidean metric. Denote by $\Lambda$ the limit set of the Kleinian group $\Gamma\cong\pi_1(M^n)$. By lemma \ref{l:c_0-estimate-on-factor} and Lemma \ref{l:integrability-lemma-of-conformal-factor}, the function
\begin{equation}
f(x)\equiv e^{w(x)},\ \forall x\in\dR^n\setminus \Lambda
\end{equation}
satisfies Assumption \ref{a:functional-assumption}.
\end{example}

\begin{proposition}\label{p:exhaustion-comparison}Given $0<\alpha<1$, assume that  $f\in L_{loc}^{\tau_2}(\dR^n)$ satisfies Assumption \ref{a:functional-assumption}, Then for every $x\in\dR^n\setminus\Lambda$, the quantity
\begin{equation}
\mathcal{L}_f(\tau)(x)\equiv\frac{\Big((-\triangle)^{\alpha}f^{\tau}\Big)(x)}{\tau\cdot f^{\tau}(x)},\ \tau>0,\end{equation}
is monotone decreasing in $\tau$.
\end{proposition}

\begin{proof}Let $f_{\epsilon}$ be the standard mollifier of $f$ for every $0<\epsilon<10^{-6}$. Proposition \ref{p:harmonic-meets-scattering} shows that 
$\mathcal{L}_{f_{\epsilon}}(\tau)(x)
$ is monotone decreasing in $\tau$. Now
applying Lemma \ref{l:basic-convergence-lemma} and Lemma \ref{l:main-convergence-lemma}, we have that for every $x\in\dR^n\setminus\Lambda$,
\begin{equation}
\mathcal{L}_f(\tau)(x)=\lim\limits_{\epsilon\to0}\mathcal{L}_{f_{\epsilon}}(\tau)(x),
\end{equation}
and hence $\mathcal{L}_f(\tau)(x)$
is monotone decreasing in $\tau$.
\end{proof}

\section{Hausdorff Dimension Bound of the Limit Set}
\label{s:hausdorff-dim-estimate}

We will prove Theorem \ref{t:dim-estimate} in this section. The basic strategy is to bound the Hausdorff dimension of the limit set $\Lambda$ by applying lemma \ref{l:hausdorff-dim-estimate}. 
So the main technical ingredients in this section are  Proposition \ref{p:int-distance-estimate} and Proposition \ref{p:2-int-distance-estimate} which establish the required integral estimates on the distance function associated to the limit set $\Lambda$.

\begin{proposition}\label{p:int-distance-estimate} Given $n\geq 3$, let $(M^n,g)$ be compact locally conformally flat with positive scalar curvature and $Q_{2\gamma}>0$ for some $1<\gamma<\max\{2,n/2\}$. Denote by $\Lambda\equiv\Lambda(\Gamma)$ the limit set of the Kleinian group $\Gamma\equiv\pi_1(M^n)$. Let $0<R<+\infty$ satisfy $\Lambda\subset B_R(0^n)$, then there exists $C_0(n,\gamma,g)$ such that
\begin{equation}
\int_{B_{R}(0^n)\setminus\Lambda}d_0^{-\frac{n+2\gamma}{2}}(x,\Lambda)\leq C_0,
\end{equation}
where $d_0$ is the Euclidean distance function. 

\end{proposition}

\begin{proof}Let $\tilde{g}=e^{2w}g_0$, then by the conformal covariance property of $P_{2\gamma}$,
\begin{equation}
 P_{2\gamma}(1)=e^{-\frac{n+2\gamma}{2}w}(-\triangle)^{\gamma}(e^{\frac{n-2\gamma}{2}w}),\end{equation}
and thus by the definition of $Q_{2\gamma}$,
\begin{equation}
Q_{2\gamma}\equiv\Big(\frac{n-2\gamma}{2}\Big)^{-1}\cdot P_{2\gamma}(1)=\Big(\frac{n-2\gamma}{2}\Big)^{-1}e^{-\frac{n+2\gamma}{2}w}(-\triangle)^{\gamma}(e^{\frac{n-2\gamma}{2}w}).\label{conformal-covariance}
\end{equation}
Applying lemma \ref{l:c_0-estimate-on-factor} and equation (\ref{conformal-covariance}), we have that
\begin{eqnarray}
\int_{B_{R}(0^n)\setminus\Lambda}Q_{2\gamma}\cdot d_0^{-\frac{n+2\gamma}{2}}(x,\Lambda)
&\leq& C(n,\gamma,g)\int_{B_{R}(0^n)\setminus\Lambda}Q_{2\gamma}\cdot e^{\frac{n+2\gamma}{2}w}\nonumber\\
&=& C(n,\gamma,g)\int_{B_{R}(0^n)\setminus\Lambda}(-\triangle)^{\gamma}(e^{\frac{n-2\gamma}{2}w}).\label{bound-int-d_0-from-above}
\end{eqnarray}

We will carry out the integration by parts to estimate the integral in \eqref{bound-int-d_0-from-above}. 
Since potentially the boundary of the domain $B_R(0^n)\setminus\Lambda$ is highly singular so that
one cannot directly apply the integration by parts. We will apply the exhaustion arguments to get around this difficulty (it was used in \cite{CHY} for the estimate in the context of positive $Q_4$).  
To this end,
denote by $\alpha\equiv\gamma-1\in(0,1)$, 
\begin{equation}U(x)\equiv\Big( (-\triangle)^{\alpha}(e^{\frac{n-2\gamma}{2}w})\Big)(x),\ x\in\dR^n,\label{exhaustion-function}\end{equation} and define
\begin{equation}
\mathcal{O}_{\lambda}\equiv\Big\{x\in B_{R}(0^n)\setminus\Lambda\Big| U(x)<\lambda\Big\}.
\end{equation}
Notice that, the function $U(x)$ is smooth away from the limit set $\Lambda$. Indeed, by the definition of $U$, \begin{equation}-\triangle U=(-\triangle)^{\gamma}e^{\frac{n-2\gamma}{2}w}=\frac{n-2\gamma}{2}\cdot Q_{2\gamma}\cdot e^{\frac{n+2\gamma}{2}w}\end{equation}
and the regularity of $U$ immediately follows.
The main point is to prove that the subsets $\mathcal{O}_{\lambda}$ effectively exhausts $B_{R}(0^n)\setminus \Lambda$ when $\lambda\to+\infty$. Precisely, we will prove that \begin{equation}
B_{R}(0^n)\setminus\Lambda=\lim\limits_{\lambda\rightarrow+\infty}U_{\lambda},
\end{equation}
and for any regular value $0<\lambda<+\infty$, the closure $\overline{\mathcal{O}_{\lambda}}$ never touches $\Lambda$ i.e., it holds that there is some constant $\delta(\lambda,n,\gamma,g)>0$ such that 
\begin{equation}
d_0(x,\Lambda)\geq \delta(\lambda,n,\gamma,g),\ \forall x\in\overline{\mathcal{O}_{\lambda}}.
\end{equation}
In particular, the level set 
\begin{equation}
E_{\lambda}\equiv\Big\{x\in B_{R}(0^n)\setminus\Lambda\Big| U(x)=\lambda\Big\}
\end{equation}
is smooth and $\Lambda=\lim\limits_{\lambda\to+\infty}E_{\lambda}$.

So it suffices to prove the claim that the function $
U(x)= (-\triangle)^{\alpha}(e^{\frac{n-2\gamma}{2}w})(x)$ uniformly bounds some power of $e^{w(x)}$ from below.
The essential technical point is to apply the monotonicity result (Proposition \ref{p:exhaustion-comparison}) 
to obtain a pointwise control.
By Proposition \ref{p:exhaustion-comparison},
\begin{eqnarray}
 \Big((-\triangle)^{\alpha}(e^{\frac{n-2\gamma}{2}w})\Big)(x)&\geq& \frac{n-2\gamma}{n-2\alpha}\cdot e^{(\alpha-\gamma)w(x)}\cdot\Big( (-\triangle)^{\alpha}e^{\frac{n-2\alpha}{2}w}\Big)(x).\label{transfer}
 \end{eqnarray}
Applying the conformal covariance property of $P_{2\alpha}$, we have that
\begin{equation}
(-\triangle)^{\alpha}e^{\frac{n-2\alpha}{2}w}=\Big(\frac{n-2\alpha}{2}\Big)\cdot Q_{2\alpha}\cdot e^{\frac{n+2\alpha}{2}w}.
\end{equation}
Plugging the above equation into inequality (\ref{transfer}),
 \begin{eqnarray} (-\triangle)^{\alpha}(e^{\frac{n-2\gamma}{2}w})(x)\geq\frac{n-2\gamma}{2}\cdot e^{\frac{n+4\alpha-2\gamma}{2}w(x)}\cdot Q_{2\alpha}(x).
\end{eqnarray}
Since we have chosen $\alpha\equiv\gamma-1\in(0,1)$, immediately 
\begin{equation}
\frac{n+4\alpha-2\gamma}{2}=\frac{n-2+2\alpha}{2}>0.
\end{equation}
Moreover, by theorem 1.2 in \cite{GQ}, there is some $c_0>0$ such that \begin{equation}Q_{2\alpha}(x)\geq c_0>0,\ \forall x\in M^n\end{equation}
provided $M^n$ is compact and $R_g>0$.
Therefore,
\begin{equation}
U(x)=(-\triangle)^{\alpha}(e^{\frac{n-2\gamma}{2}w})(x)\geq c_0\cdot\frac{n-2\gamma}{2}\cdot e^{\frac{n-2+2\alpha}{2}w(x)}.\label{effective-control}\end{equation}
So the claim has been proven.

Now we are ready to implement the integration by parts.
In the above notations, it holds that $(-\triangle)^{\gamma}=(-\triangle)\circ(-\triangle)^{\alpha}$. 
Let us estimate the following integral 
over $\mathcal{O}_{\lambda}$. Let $\lambda>0$ be a generic regular value such that  the level set $E_{\lambda}$ is a smooth $(n-1)$-dimensional submanifold in $B_R(0^n)\setminus\Lambda$. Let $\nu$ be the outward unit normal vector field on $\partial \mathcal{O}_{\lambda}$, then integration by parts gives that
\begin{eqnarray}
\int_{ \mathcal{O}_{\lambda}}(-\triangle)^{\gamma}(e^{\frac{n-2\gamma}{2}w})&=&\int_{ \mathcal{O}_{\lambda}}-\triangle U\nonumber\\
&=&-\int_{\partial  \mathcal{O}_{\lambda}}\langle\nabla U, \nu\rangle\nonumber\\
&=&-\int_{E_{\lambda}}\langle\nabla U, \nu\rangle-\int_{\partial B_{R}(0^n)} \langle\nabla U, \nu\rangle.\end{eqnarray}
On the level set $E_{\lambda}$ we have that $\nu=\frac{\nabla U}{|\nabla U|}$, which implies that
\begin{eqnarray}\int_{\mathcal{O}_{\lambda}}(-\triangle)^{\gamma}(e^{\frac{n-2\gamma}{2}w})&=&-\int_{E_{\lambda}}|\nabla U|-\int_{\partial B_{R}(0^n)} \langle\nabla U, \nu\rangle\nonumber\\&\leq& C(n,\gamma,g).
\end{eqnarray}
Taking the limit $\lambda\rightarrow+\infty$, we have that
\begin{equation}
\int_{B_{R}(0^n)\setminus\Lambda}(-\triangle)^{\gamma}(e^{\frac{n-2\gamma}{2}w})=\lim\limits_{\lambda\rightarrow+\infty}\int_{ \mathcal{O}_{\lambda}}(-\triangle)^{\gamma}(e^{\frac{n-2\gamma}{2}w})\leq C(n,\gamma,g)<+\infty.
\end{equation}
Plugging the above estimate and the assumption $Q_{2\gamma}\geq c_1>0$ into (\ref{bound-int-d_0-from-above}),
\begin{equation}
\int_{B_{R}(0^n)\setminus\Lambda}d_0^{-\frac{n+2\gamma}{2}}(x,\Lambda)\leq C(n,\gamma,g) <\infty.\end{equation}
So the proof of the proposition is complete.

\end{proof}

Now we relax the assumption $Q_{2\gamma}>0$ in Proposition \ref{p:int-distance-estimate}
to $Q_{2\gamma}\geq0$, then we will prove the following.

\begin{proposition}
\label{p:2-int-distance-estimate}
Given $n\geq 3$, let $(M^n,g)$ be compact locally conformally flat with positive scalar curvature and $Q_{2\gamma}\geq0$ for some $1<\gamma<\max\{2,n/2\}$. Denote by $\Lambda\equiv\Lambda(\Gamma)$ the limit set of the Kleinian group $\Gamma\equiv\pi_1(M^n)$. Let $0<R<+\infty$ satisfy $\Lambda\subset B_R(0^n)$, then for every $\epsilon>0$, there exists $C_0(\epsilon,n,\gamma,g)>0$ such that
\begin{equation}
\int_{B_{R}(0^n)\setminus\Lambda}d_0^{-(\frac{n+2\gamma}{2}-\epsilon)}(x,\Lambda)\leq C_0(\epsilon,n,\gamma, g),\label{eps-int}
\end{equation}
where $d_0$ is the Euclidean distance function. 

\end{proposition}

\begin{remark}
The integral in \eqref{eps-int} is infinite if $\epsilon=0$, and hence the estimate \eqref{eps-int} is rather sharp. That is, $d_0$ might be not integrable around the set $\Lambda$ when $\epsilon=0$. Indeed, in Example \ref{e:hyperbolic}, $n=5$ and $\gamma=\frac{3}{2}$, by choosing the cylindrical coordinates around the circle $\Lambda\equiv\mathbb{S}^1$, we can see that the function $d_0^{-4+\epsilon}$ is integrable around $\Lambda$ for any $\epsilon>0$, while it is not integrable when $\epsilon=0$.  
\end{remark}

\begin{proof}
Similar to the proof of Proposition \ref{p:int-distance-estimate}, we will prove the estimate \eqref{eps-int} by applying
the exhaustion arguments.
The conformal covariance property of the operator $P_{2\gamma}$ implies that\begin{equation}
(-\triangle)^{\gamma}e^{\frac{n-2\gamma}{2}w}=\Big(\frac{n-2\gamma}{2}\Big)\cdot Q_{2\gamma}\cdot e^{\frac{n+2\gamma}{2}w}.\label{p-conformal-cov}
\end{equation}
Let $\alpha\equiv\gamma-1\in(0,1)$ the function
$
U(x)$ was defined by \eqref{exhaustion-function} and the behavior of its sublevel set $\mathcal{O}_{\lambda}$ was studied in the proof of Proposition \ref{p:int-distance-estimate}, which enables us to implement the integration by parts as follows. 
Since it has been shown in the proof of Proposition \ref{p:int-distance-estimate} that the smooth sublevel set $\mathcal{O}_{\lambda}$ effectively exhausts $B_R(0^n)\setminus\Lambda$ as $\lambda\to+\infty$ and level set $E_{\lambda}$ approximates $\Lambda$, it suffices to prove the integral  estimate on the open set $\mathcal{O}_{\lambda}$ for sufficiently large $\lambda$. 

For every $\beta\in\dR$, by \eqref{p-conformal-cov} and integration by parts,
\begin{eqnarray} 
0&\leq&(\frac{n-2\gamma}{2})\int_{\mathcal{O}_{\lambda}}Q_{2\gamma}\cdot e^{(\frac{n+2\gamma}{2}+\beta)w}\nonumber\\&=&\int_{\mathcal{O}_{\lambda}}e^{\beta w}\cdot(-\triangle)^{\gamma}e^{\frac{n-2\gamma}{2}w}\nonumber\\
&=&\int_{\mathcal{O}_{\lambda}}e^{\beta w}\cdot(-\triangle U)\nonumber\\
&=&-\int_{\mathcal{O}_{\lambda}}\triangle(e^{\beta w})
\cdot U+\int_{E_{\lambda}}U\cdot\Big\langle\nabla e^{\beta w},\frac{\nabla U}{|\nabla U|}\Big\rangle-\int_{E_{\lambda}}e^{\beta w}|\nabla U| +\int_{\partial B_{R}(0^n)}B,\label{beta-int-by-parts}
\end{eqnarray}
where 
\begin{equation}
B\equiv U\cdot\Big\langle\nabla e^{\beta w},\frac{\partial}{\partial r}\Big\rangle-e^{\beta w}\Big\langle\nabla U,\frac{\partial}{\partial r}\Big\rangle.
\end{equation}
Immediately,
\begin{equation}
\Big|\int_{\partial B_{R}(0^n)}B\Big|\leq C(\beta, n, \gamma, g).\label{bound-on-B}
\end{equation}

Next we will prove the claim that if $\beta<0$, then
\begin{equation}
\int_{E_{\lambda}}U\cdot\Big\langle\nabla e^{\beta w},\frac{\nabla U}{|\nabla U|}\Big\rangle< 0.\label{negativity-of-product}
\end{equation}
It is straightforward that for $\beta<0$ and $R_g>0$,
\begin{equation}
\triangle (e^{\beta w})=-\beta\cdot\frac{R_g}{6}\cdot e^{(\beta+2)w}
+\beta\Big(\beta-\frac{n-2}{2}\Big)|\nabla w|^2e^{\beta w}>0.\label{laplace-beta}
\end{equation}
Now we define the open set for $\mu>0$ and $\lambda>0$,
\begin{equation}
\Sigma_{\lambda,\mu}\equiv\{x\in B_R(0^n)\setminus\Lambda|w(x)<\mu,\ U(x)>\lambda\}.
\end{equation}
Let $F_{\mu}\equiv\{x\in B_R(0^n)\setminus\Lambda|w(x)=\mu\}$ and hence 
$\Sigma_{\lambda,\mu}$ is smooth for generic $\lambda$, $\mu$ and 
\begin{equation}
\partial\Sigma_{\lambda,\mu}=F_{\mu}\cup E_{\lambda}
\end{equation}
We will prove \eqref{negativity-of-product} by the integration by parts on the domain $\Sigma_{\lambda,\mu}$. In fact,
\begin{eqnarray}
0<\int_{\Sigma_{\lambda,\mu}}\triangle(e^{\beta w})&=&\int_{F_{\mu}}\langle\nabla e^{\beta w},\nu\rangle+\int_{E_{\lambda}}\langle\nabla e^{\beta w},\nu\rangle\nonumber\\
&=&\beta\int_{F_{\mu}}e^{\beta w}|\nabla w|-\int_{\Sigma_{\lambda}}\Big\langle\nabla e^{\beta w},\frac{\nabla U}{|\nabla U|}\Big\rangle,
\end{eqnarray}
where $\nu$ is the outward normal vector field. 
Since $\beta<0$, so inequality \eqref{negativity-of-product} immediately follows.
Next, by \eqref{negativity-of-product}, 
\eqref{bound-on-B} and 
\eqref{beta-int-by-parts}, it holds that
\begin{equation}
\int_{\mathcal{O_{\lambda}}}\triangle(e^{\beta w})
\cdot U\leq C(\beta,n,\gamma,g).\label{bound-on-product}
\end{equation}
Since we have proved in Proposition \ref{p:int-distance-estimate} that (see \eqref{effective-control}),
\begin{equation}
U(x)=(-\triangle)^{\alpha}(e^{\frac{n-2\gamma}{2}w})(x)\geq c_0\cdot\frac{n-2\gamma}{2}\cdot e^{\frac{n-2+2\alpha}{2}w(x)}.\end{equation}
Now combined with  
\eqref{laplace-beta}, then for every $\beta>0$, there are constants $C_1(\beta,n,\gamma,g)>0$ and $C_2(\beta,n,\gamma,g)>0$ such that
\begin{equation}
\triangle (e^{\beta w})\cdot U\geq C_1\cdot e^{(\frac{n+2+2\alpha}{2}+\beta)w}+C_2\cdot |\nabla w|^2\cdot e^{\frac{n-2+2\alpha+2\beta}{2} w}.\label{pointwise-comparison}
\end{equation}
Therefore, let $\beta=-\epsilon$, then by \eqref{bound-on-product} and \eqref{pointwise-comparison},
\begin{equation}
\int_{\mathcal{O}_{\lambda}}
e^{(\frac{n+2\gamma}{2}-\epsilon)w}\leq C(\epsilon,n,\gamma,g).
\end{equation}
So we can finish the proof by applying the exhaustion property and lemma \ref{l:c_0-estimate-on-factor}.

\end{proof}

With the above technical preparations, we are in a position to prove Theorem \ref{t:dim-estimate}.

\begin{proof}
[Proof of Theorem \ref{t:dim-estimate}] First we prove the first case which states that $\dim_{\mathcal{H}}(\Lambda)<\frac{n-2\gamma}{2}$ if $Q_{2\gamma}>0$.
 By Proposition \ref{p:int-distance-estimate} and lemma \ref{l:hausdorff-dim-estimate}, we have that
\begin{equation}
\dim_{\mathcal{H}}(\Lambda)\leq\frac{n-2\gamma}{2}.
\end{equation}
Then standard continuity argument gives that there exists $\epsilon>0$ such that for every $\gamma'\in(\gamma,\gamma+\epsilon)$ it holds that $Q_{2\gamma'}>0$ which implies that
\begin{equation}
\dim_{\mathcal{H}}(\Lambda)\leq\frac{n-2\gamma'}{2}<\frac{n-2\gamma}{2}.
\end{equation}

Next we will prove that $\dim_{\mathcal{H}}(\Lambda)\leq\frac{n-2\gamma}{2}$ if the assumption is relaxed to $Q_{2\gamma}\geq0$. In this case, applying Proposition \ref{p:2-int-distance-estimate} and lemma \ref{l:hausdorff-dim-estimate}, we obtain that, for every $\epsilon>0$,
\begin{equation}
\dim_{\mathcal{H}}(\Lambda)\leq\frac{n-2\gamma}{2}+\epsilon,
\end{equation}
which implies that
\begin{equation}
\dim_{\mathcal{H}}(\Lambda)\leq\frac{n-2\gamma}{2}
\end{equation}
by taking $\epsilon\to0$.

So the proof is done.
\end{proof}

\section{The Critical Case: $Q_3$ Curvature on $3$-Dimensional Manifolds}\label{s:critical-case}

In this section, we will study 
the limit set $\Lambda(\Gamma)$
in the critical case that $2\gamma=n=3$. 
We will prove Theorem \ref{t:3d-rigidity-theorem} which gives a geometric rigidity result in the critical case. 
The basic idea of the proof is similar to that of Theorem \ref{t:dim-estimate}.
The new point is Proposition \ref{p:limiting-exhaustion-comparison} which can be viewed as a limiting version of Proposition \ref{p:exhaustion-comparison}.

\begin{proposition}\label{p:limiting-exhaustion-comparison} Let $\Lambda\subset\dR^3$
be a closed subset of zero measure and let $e^w$ satisfy Assumption \ref{a:functional-assumption}.
Then for every $x\in\dR^3\setminus\Lambda$,
\begin{equation}
((-\triangle)^{1/2}w)(x)\geq e^{-w(x)}((-\triangle)^{1/2}e^{w})(x).\label{limiting-monotonicity}
\end{equation}
\end{proposition}

\begin{proof}Denote by $f\equiv e^{w}$ and let $f_{\epsilon}$ the standard mollifier of $f$ in the sense of (\ref{def-mollifier}).
Let $1<\gamma<3/2$, by Proposition \ref{p:harmonic-meets-scattering}, then for every $x\in\dR^3$
\begin{equation}
\frac{\tau_2}{\tau_1}\cdot \Big((-\triangle)^{\gamma-1}(f_{\epsilon})^{\tau_1}\Big)(x)\geq (f_{\epsilon})^{\tau_1-\tau_2}(x)\cdot\Big((-\triangle)^{\gamma-1}(f_{\epsilon})^{\tau_2}\Big)(x),\label{ineq-before-taking-limit}
\end{equation}
where $\tau_1=\frac{3-2\gamma}{2}$ and $\tau_2=\frac{5-2\gamma}{2}$.
At the first stage, we will prove that \begin{equation}
\Big((-\triangle)^{1/2}(\log (f_{\epsilon}))\Big)(x)\geq (f_{\epsilon})^{-1}(x)\cdot(-\triangle)^{1/2}(f_{\epsilon})(x).\label{smooth-limiting-exhaustion-estimate}
\end{equation}

To this end, let $\gamma\to\frac{3}{2}$ then $\tau_1\to0$ and $\tau_2\to1$, by \eqref{ineq-before-taking-limit}, so it turns out that
\begin{eqnarray}
\lim\limits_{\gamma\to\frac{3}{2}}\frac{\tau_2}{\tau_1}\cdot \Big((-\triangle)^{\gamma-1}(f_{\epsilon})^{\tau_1}\Big)(x)\geq (f_{\epsilon})^{-1}(x)\cdot(-\triangle)^{1/2}(f_{\epsilon})(x).\label{approaching-to-the-critical}
\end{eqnarray}
Notice that, the left hand side can be written as
 \begin{eqnarray}
\lim\limits_{\gamma\to\frac{3}{2}}\frac{\tau_2}{\tau_1}\cdot \Big((-\triangle)^{\gamma-1}(f_{\epsilon})^{\tau_1}\Big)(x)&=&\lim\limits_{\gamma\to\frac{3}{2}} (-\triangle)^{\gamma-1}\Big(\frac{\tau_2}{\tau_1}\cdot ((f_{\epsilon})^{\tau_1}-1)\Big)(x).
 \end{eqnarray}
It follows from  the Taylor expansion in $\tau_1$ that 
\begin{equation}(f_{\epsilon})^{\tau_1}-1=e^{\tau_1\cdot\log(f_{\epsilon})}-1=\tau_1\cdot\log (f_{\epsilon})+\sum\limits_{k=2}^{\infty}\frac{(\tau_1\log (f_{\epsilon}))^k}{k!},\end{equation} which implies that
\begin{equation}
\lim\limits_{\gamma\to\frac{3}{2}}\frac{\tau_2}{\tau_1}\cdot ((f_{\epsilon})^{\tau_1}-1)=\log (f_{\epsilon}).
\end{equation}
Consequently,
\begin{equation}
 \lim\limits_{\gamma\to\frac{3}{2}}\frac{\tau_2}{\tau_1}\cdot \Big((-\triangle)^{\gamma-1}(f_{\epsilon})^{\tau_1}\Big)(x)=\Big((-\triangle)^{1/2}\log (f_{\epsilon})\Big)(x).
\end{equation}
Then (\ref{approaching-to-the-critical}) implies that,
\begin{equation}
\Big((-\triangle)^{1/2}\log (f_{\epsilon})\Big)(x)\geq   (f_{\epsilon})^{-1}(x)\cdot(-\triangle)^{1/2}(f_{\epsilon})(x).
\end{equation}
So the proof of (\ref{smooth-limiting-exhaustion-estimate}) is done.

Next, applying Lemma \ref{l:basic-convergence-lemma} and (5) in Lemma \ref{l:main-convergence-lemma}, we have that for every $x\in\dR^3\setminus\Lambda$,
\begin{equation}
((-\triangle)^{1/2}w)(x)\geq e^{-w(x)}((-\triangle)^{1/2}e^{w})(x).
\end{equation}
Now the proof of the proposition is complete.
\end{proof}

As what we did in Section \ref{s:hausdorff-dim-estimate}, we are ready to prove Theorem \ref{t:3d-rigidity-theorem}.

\begin{proof}
[Proof of Theorem \ref{t:3d-rigidity-theorem}] 
The proof of the theorem consists of the following two primary steps.

The first step is to prove that 
$\dim_{\mathcal{H}}(\Lambda(\Gamma))=0
$ under the assumption of $R_g>0$ and $Q_3\geq0$ on $(M^3,g)$.
By lemma \ref{l:c_0-estimate-on-factor} and lemma \ref{l:hausdorff-dim-estimate}, it suffices to show that for every $\epsilon>0$,
\begin{equation}
\int_{B_{R}(0^3)\setminus\Lambda}e^{(3-\epsilon)w(x)}dx\leq C(\epsilon,g).\label{3d-conformal-factor-estimate}
\end{equation}

We will prove \eqref{3d-conformal-factor-estimate} by the similar exhaustion arguments and the integration by parts as in the proof of Proposition \ref{p:2-int-distance-estimate}. 
The crucial point is to effectively approximate $B_R(0^3)\setminus\Lambda$ by a smooth domain.
Precisely, define
\begin{equation}
U(x)\equiv\Big((-\triangle)^{1/2}w\Big)(x), \  x\in\dR^3,
\end{equation}
and given $\lambda\in\dR$ we define
\begin{equation}
\mathcal{O}_{\lambda}\equiv\Big\{x\in B_R(0^3)\setminus\Lambda\Big| U(x)<\lambda\Big\}
\end{equation}
and
\begin{equation}
E_{\lambda}\equiv\{x\in B_R(0^3)\setminus\Lambda\Big| U(x)=\lambda\}.
\end{equation}
So we will show that $B_R(0^3)\setminus\Lambda$ can be effectively approximated by $\mathcal{O}_{\lambda}$. To establish such an approximation, from the proof of Proposition \ref{p:int-distance-estimate} and Proposition \ref{p:2-int-distance-estimate}, it suffices to prove the following pointwise estimate: there exists a constant $c_0(g)>0$ such that for every $x\in\dR^3\setminus\Lambda$,
\begin{equation}
U(x)\geq c_0e^{w(x)}.\label{3d-pointwise-bound}
\end{equation}
The proof of \eqref{3d-pointwise-bound} follows from the limiting monotonicity inequality \eqref{limiting-monotonicity}.
In fact,
let $\hat{g}$ be the Riemannian metric on the universal cover of $M^3$, and thus $\hat{g}=e^{2w}g_0$. Denote by $\widehat{Q}_1\equiv Q_{1,\hat{g}}$, then the conformal covariance property implies that 
\begin{equation}
\widehat{Q}_1\cdot e^{2w}=(-\triangle)^{1/2}e^w.
\end{equation}
On the other hand, by Proposition \ref{p:limiting-exhaustion-comparison}, for every $x\in\dR^3\setminus\Lambda$, 
\begin{equation}
((-\triangle)^{1/2}w)(x)\geq e^{-w(x)}((-\triangle)^{1/2}e^w)(x).
\end{equation}
Combining the above inequalities, 
\begin{equation}
((-\triangle)^{1/2}w)(x)\geq  e^{-w(x)}\cdot(-\triangle)^{1/2}(e^w)(x)=e^{w(x)}\cdot \widehat{Q}_1(x).
\end{equation}
Since $R_g>0$, applying a result in \cite{GQ}, we have that $Q_1>0$ everywhere on $M^3$. By assumption, $M^3$ is compact and so there is some $c_0(g)>0$ such that 
\begin{equation}
Q_1(x)\geq c_0>0,\ \forall x\in M^3,
\end{equation}
which implies that 
\begin{equation}
\widehat{Q}_1(x)\geq c_0>0,\ \forall x\in \widetilde{M}^3.
\end{equation}
Therefore,
\begin{equation}
U(x)=((-\triangle)^{1/2}w)(x)\geq c_0\cdot e^{w(x)}.
\end{equation}

With the above estimate  we are in a position to carry out the integration by parts to prove inequality (\ref{3d-conformal-factor-estimate}).
Applying the condition $Q_3\geq0$ and the same argument as in the proof of Proposition \ref{p:2-int-distance-estimate}, we have that for every $0<\epsilon<1$,
\begin{eqnarray}
\int_{\mathcal{O_{\lambda}}}\triangle(e^{-\epsilon w})
\cdot U\leq C(\epsilon,n,\gamma,g).\label{3d-bound-on-product}
\end{eqnarray}
Next, since
\begin{equation}
\triangle (e^{-\epsilon w})=\epsilon\cdot\frac{R_g}{6}\cdot e^{(2-\epsilon)w}
+\epsilon\Big(\frac{n-2}{2}-\epsilon\Big)|\nabla w|^2e^{-\epsilon w},\label{laplace-beta}
\end{equation}
and
\begin{equation}
U(x)\geq c_0\cdot e^{w},
\end{equation}
so plugging the above into inequality \eqref{3d-bound-on-product},
\begin{equation}
\int_{\mathcal{O}_{\lambda}}e^{(3-\epsilon)w}\leq C(\epsilon,g)
\end{equation}
and thus \eqref{3d-conformal-factor-estimate} has been proved. Therefore, \begin{equation}\dim_{\mathcal{H}}(\Lambda)=0.\label{0-dim}\end{equation}
So we finish the proof of the first step.

The second step is to analyze the geometry of $(M^3,g)$ provided \eqref{0-dim}. An immediate consequence of \eqref{0-dim} is that $\Gamma\equiv\pi_1(M^3)$ is elementary. Since we have assumed $R_g>0$, so $M^3$ is never covered by $\mathbb{T}^3$. Moreover, $\Gamma$ is torsion-free, so
there are two cases:
\begin{enumerate}
\item 
 $M^3$ is conformal to the round sphere $(\mathbb{S}^3,g_1)$ of constant curvature $1$. The corresponding hyperbolic filling-in is $(\dH^4,g_+)$.

\item $M^3$ is conformal to the Riemannian product space 
$(\mathbb{S}^1\times\mathbb{S}^2, dt^2\oplus h_1)$ with the metric $h_1$ of constant Gaussian curvature $1$.
The  corresponding hyperbolic filling-in associated to the latter case is $(\dR^3\times\mathbb{S}^1,g_+)$.
\end{enumerate}

Next we will apply Chern-Gauss-Bonnet formula in the context of conformally compact Einstein manifolds to 
obtain the desired rigidity. In our context, $(X^4,g_+)$ is a hyperbolic manifold with a conformal infinity $(M^3,g)$ such that $Q_3(g,g_+)\geq0$, 
then Chern-Gauss-Bonnet formula \eqref{chern-gauss-bonnet} implies that
\begin{equation}
4\pi^2\chi(X^4)=\int_{M^3}Q_3\dvol_g.
\end{equation}
First, if $Q_3(x)>0$ for some $x\in M^3$, then
\begin{equation}
\int_{M^3}Q_3\dvol_g>0
\end{equation}
and hence $\chi(X^4)>0$. So it corresponds to case (1). Now we focus on the case $Q_3\equiv0$ on $M^3$, then immediately $\chi(X^4)=0$ and this corresponds to case (2).
As the last step, we will show that $(M^3,g)$
is isometric to $(\mathbb{S}^1\times\mathbb{S}^2, dt^2\oplus h_r)$ for a round metric $h_r$ on $\mathbb{S}^2$. We have proved that $(M^3,g)$ is conformal to the Riemannian product space $(\mathbb{S}^1\times\mathbb{S}^2, \sigma_r)$ with $\sigma_r\equiv dt^2\oplus h_r$, so we can write 
\begin{equation}
g=e^{2w}\sigma_r,
\end{equation}
and then the following conformal covariance property holds,
\begin{equation}
P_3[\sigma_r](w)+Q_3(\sigma_r)=Q_3(g)e^{3w}.
\end{equation}
Since $Q_3(\sigma_r)\equiv0$ and by assumption $Q_3(g)\equiv0$, so 
\begin{equation}
P_3[\sigma_r](w)\equiv0.
\end{equation}
By theorem 1.2 in \cite{CaCh}), in our case $Q_3(\sigma_r)\equiv0$ and $R_{\sigma_r}>0$, so the operator $P_3[\sigma_r]$ is nonnegative and $\Ker(P_3[\sigma_r])=\dR$. Therefore $w$ is a constant and $g$ is a rescaling of the standard metric $\sigma_r$.
 So the proof is done.

\end{proof}

\section{Examples}

We will present in this section several examples with concrete Kleinian groups and we will compute the associated nonlocal curvature $Q_{2\gamma}$. 
In addition, we will give an example which shows both the estimate in Theorem \ref{t:dim-estimate} and the assumption in 
Theorem \ref{t:topological-rigidity} are in fact optimal.
Moreover, we will show that Theorem \ref{t:3d-rigidity-theorem}
does not hold in the general setting of conformally compact Einstein manifolds.

Example \ref{e:Q-curvature-sphere} and Example \ref{e:Q-curvature-quotient-of-cylinder} are rather standard 
spaces with positive $Q_{2\gamma}$
curvature and the associated limit set is isolated. The detailed computations can be found in 
in \cite{DG}.

\begin{example}[Standard sphere]\label{e:Q-curvature-sphere}
Given any $n\geq 3$, consider the sphere 
$(\mathbb{S}^{n},g_1)$ with $g_1=(\frac{2}{1+|x|^2})^2g_0$. Immediately, $\sec_{g_1}\equiv 1$ and $R_{g_1}\equiv n(n-1)$. 
Then for every  $\gamma\in(0,n/2)$, by the conformal covariance property of $P_{2\gamma}$, we can compute that
\begin{equation}P_{2\gamma}(1)\equiv\frac{\Gamma(\frac{n+2\gamma}{2})}{\Gamma(\frac{n-2\gamma}{2})}>0.\end{equation}
It follows that for every $\gamma\in(0,n/2)$,
\begin{equation}
Q_{2\gamma}\equiv(\frac{n-2\gamma}{2})^{-1}P_{2\gamma}(1)>0.
\end{equation}
Let $n=3$ and $\gamma=\frac{n}{2}=\frac{3}{2}$,
\begin{equation}
Q_{3}\equiv\lim\limits_{\gamma\to 3/2}(\frac{3-2\gamma}{2})^{-1}P_{2\gamma}(1)=\lim\limits_{\gamma\to 3/2}\frac{\Gamma(\frac{3+2\gamma}{2})}{\frac{3-2\gamma}{2}\cdot\Gamma(\frac{3-2\gamma}{2})}=\Gamma(3)=2.
\end{equation}

\end{example}

\begin{example}[Standard cylinder]\label{e:Q-curvature-quotient-of-cylinder}
Fix any $n\geq 3$, consider the compact manifold  $(M^n,h)\equiv(\mathbb{S}^{n-1}\times\mathbb{S}^{1},h)$ with the standard product metric such that $R_h\equiv(n-1)(n-2)$. We consider the case that the filling-in of $\mathbb{S}^{n-1}\times\mathbb{S}^1$ is the hyperbolic manifold $\dH^{n+1}/\dZ$ which is diffeomorphic to $\dR^{n}\times\mathbb{S}^1$. To understand the behavior of $P_{2\gamma}$, we look at the universal cover
$(\widetilde{M^n},\tilde{h})\equiv(\mathbb{S}^{n-1}\times\dR^1,\tilde{h})$, where the covering metric $\tilde{h}$
can be written in terms of polar coordinates in $\dR^n$,
$\tilde{h}\equiv\frac{g_{\dR^n}}{r^2}$.
So standard computations imply that for every $\gamma\in(0,n/2)$, \begin{equation}
P_{2\gamma}(1)= 2^{2\gamma}\frac{\Gamma^2(\frac{n+2\gamma}{4})}{\Gamma^2(\frac{n-2\gamma}{4})}>0,
\end{equation}
and hence by definition,
\begin{equation}
Q_{2\gamma}\equiv(\frac{n-2\gamma}{2})^{-1}P_{2\gamma}(1)>0,\ \gamma\in(0,n/2).
\end{equation}
Then
it follows that in the critical case, namely, $n=3$ and $\gamma=\frac{3}{2}$,
\begin{equation}
Q_{3}\equiv\lim\limits_{\gamma\to 3/2}(\frac{3-2\gamma}{2})^{-1}P_{2\gamma}(1)=\lim\limits_{\gamma\to 3/2}\frac{2^{2\gamma}\cdot \Gamma^2(\frac{3+2\gamma}{4})}{\frac{3-2\gamma}{2}\cdot\Gamma^2(\frac{3-2\gamma}{4})}=0.
\end{equation}

\end{example}

In the next example, the Kleinian group $\Gamma$ has limit set $\Lambda(\Gamma)$ of positive Hausdorff dimension. From this example, one can see that the dimension estimate in Theorem \ref{t:dim-estimate} is sharp, and the lower bound $\gamma=\frac{3}{2}$ in Theorem \ref{t:topological-rigidity} is also optimal.

\begin{example}[Limit set of positive dimension]\label{e:hyperbolic}
Let $(M^5,\omega)\equiv (\mathbb{S}^3\times \Sigma^2, g_1\oplus g_{\Sigma^2})$ be a Riemannian product space,
where $(\mathbb{S}^3,g_1)$ is the $3$-sphere with $\sec_{g_1}\equiv 1$ and  $(\Sigma^2,h_{\Sigma^2})$ is a closed hyperbolic surface with $\sec_{h_{\Sigma^2}}\equiv -1$.
Then by standard calculations, the scalar curvature is constant $R_{\omega}\equiv 4$ and $(M^5,\omega)$ is a locally conformally flat manifold. Indeed, the universal cover $(\widetilde{M^5},\tilde{\omega})\equiv(\mathbb{S}^3\times\mathbb{H}^2$, $g_1\oplus h_{-1})$ has
$\sec_{g_1}\equiv1$ and $\sec_{h_{-1}}\equiv-1$, which implies that
\begin{equation}\tilde{\omega}=g_1\oplus h_{-1}=g_1+\frac{dy^2+dx^2}{y^2}=\frac{g_{\dR^5}}{y^2}, \ y>0.\label{conformal-to-Euclidean}
\end{equation}
Therefore, $(\widetilde{M^5},\tilde{\omega})$
is conformally flat and $\widetilde{M^5}\overset{homeo}{\cong}\mathbb{S}^5\setminus\mathbb{S}^1$.
It is a standard fact from hyperbolic geometry that $\Gamma\equiv\pi_1(M^5)=\pi_1(\Sigma^2)$ is actually of Fuchsian type
such that \begin{equation}\dim_{\mathcal{H}}(\Lambda(\Gamma))=1=\frac{5-3}{2},\end{equation} and $M^5\overset{homeo}{\cong}(\mathbb{S}^5\setminus\mathbb{S}^1)/\Gamma$. 

Next, the universal covering space $(\widetilde{M^5},\tilde{\omega})$
can be explicitly written as the conformal infinity of the $6$-dimensional hyperbolic space 
$(\mathbb{H}^{6},g_{\mathbb{H}^6})$ with $\sec_{g_{\mathbb{H}^6}}\equiv-1$. In fact, $g_{\mathbb{H}^6}$  can be represented as follows,
\begin{equation}
g_{wp}=ds^2+\sinh^2(s)g_1+\cosh^2(s)g_{-1},
\end{equation}
where $g_1$ is the metric on $\mathbb{S}^3$ with $\sec_{g_{1}}\equiv1$ and $g_{-1}$ is the metric on $\mathbb{H}^2$ with $\sec_{g_{-1}}\equiv-1$. 
Under the coordinate change $s=-\log(\frac{r}{2})$, then we have that
\begin{equation}
g_{wp}=\frac{dr^2+(1-\frac{r^2}{4})^2g_{1}+(1+\frac{r^2}{4})^2g_{-1}}{r^2}.
\end{equation}
Indeed, with the above geodesic defining function $r$, $(\widetilde{M^5},\widetilde{\omega})$
can be viewed as the conformal infinity of $(\mathbb{H}^6,g_{-1})$.

Now we are in a position to understand the fractional order curvatures of $(M^5,\omega)$.
By the standard computations in Fourier transform (see \cite{GonMaSi} for details), for the above conformal factor $y$ in (\ref{conformal-to-Euclidean}), one can see that for $0<\gamma<\frac{3}{2}$,
\begin{equation}
(-\triangle)^{\gamma}y^{-\frac{5-2\gamma}{2}}=c_0(\gamma)\cdot y^{-\frac{5+2\gamma}{2}},\end{equation}
where 
\begin{equation}
c_0(\gamma)\equiv \frac{\pi^{2\gamma-4}}{16}\cdot\frac{\Gamma(\frac{3+2\gamma}{4})\cdot\Gamma(\frac{5+2\gamma}{4})}{\Gamma(\frac{5-2\gamma}{4})\cdot\Gamma(\frac{3-2\gamma}{4})}.
\end{equation}
Since $\tilde{\omega}=\frac{g_{\dR^5}}{y^2}$, by the conformal covariance property of $P_{2\gamma}\equiv P_{2\gamma}[\tilde{\omega}]$, we have that
 \begin{equation}P_{2\gamma}(1)=y^{\frac{5+2\gamma}{2}}(-\triangle)^{\gamma}y^{-\frac{5-2\gamma}{2}}=c_0(\gamma),\  \forall 0<\gamma<\frac{3}{2}.\end{equation} 
 Hence by definition,  \begin{equation}Q_{2\gamma}\equiv\frac{2c_0(\gamma)}{5-2\gamma}>0,\ \forall 0<\gamma<\frac{3}{2},\end{equation} which implies that \begin{equation}Q_3=\lim\limits_{\gamma\to3/2}Q_{2\gamma}=\lim\limits_{\gamma\to3/2}c_0(\gamma)=0.\end{equation}
Generally, by Theorem \ref{t:dim-estimate}, for every $\frac{3}{2}\leq \gamma<2$, 
$M^5$ does not admit any conformally flat metric with $Q_{2\gamma}>0$ and $R>0$.  Moreover, Chang-Hang-Yang's work in \cite{CHY} shows that $M^5$ does not admit any locally conformally flat metric 
with $R_h>0$ and $Q_4>0$. 

\end{example}

\begin{example}
[AdS-Schwarzschild Space] Consider the AdS-Schwarzschild space $(\dR^2\times\mathbb{S}^2,g_+)$\label{e:ads-schwarzschild}
with
\begin{equation}
g_+=\frac{dr^2}{V(r)}+V(r)dt^2+r^2g_{\mathbb{S}^2},
\end{equation}
where
\begin{equation}
V(r)\equiv 1+r^2-\frac{2m}{r},\ m>0.
\end{equation}
Notice that $g_+$ is a hyperbolic metric with $\sec_{g_+}\equiv-1$ only if $m=0$.
So for each $m>0$,  the space $(\dR^2\times\mathbb{S}^2,g_+)$ is an Einstein manifold (but non-hyperbolic) with constant Ricci curvature $-3$. Moreover, its conformal infinity is diffeomorphic to $\mathbb{S}^1\times\mathbb{S}^2$.

Now we consider a special case $0<m< 1$, then the compactified space given by the conformal metric $r^{-2}g_+$ is diffeomorphic to $\dD^2\times\mathbb{S}^2$ such that the boundary is isometric to $M^3\equiv\mathbb{S}^1(\lambda)\times\mathbb{S}^2$, where $\mathbb{S}^1(\lambda)$ is the circle of radius $\lambda$ and $\mathbb{S}^2$ is the round sphere with constant curvature $1$.  

We will show that the nonlocal curvature $Q_3$ is a positive constant on the above $M^3$ associated to the metric $g_+$ with $0<m<1$. 
$Q_3$ can be computed in terms of the renormalized volume $\mathcal{V}(\dR^2\times\mathbb{S}^2,g_+)$ in this special case. Indeed, a classical formula due to C. Fefferman and R. Graham \cite{FG} shows that
\begin{equation}
3\mathcal{V}(\dR^2\times\mathbb{S}^2,g_+)=\int_{\mathbb{S}^{1}(\lambda)\times\mathbb{S}^2}Q_3.
\end{equation}
On the one hand, explicit computations imply that $\mathcal{V}(\dR^2\times\mathbb{S}^2,g_+)>0$ if $0<m<1$ (see \cite{CQY2} for more details of the computations).
This implies that
\begin{equation}
\int_{\mathbb{S}^1(\lambda)\times\mathbb{S}^2}Q_3>0.
\end{equation}
On the other hand, the symmetry property of $g_+$ actually implies that $Q_3$ is a constant on the conformal infinity. Therefore, $Q_3$ is a positive constant.  

So this example shows that for $(\mathbb{S}^1\times \mathbb{S}^2,h)$, there exists 
a conformally compact Einstein filling-in $(\dR^2\times \mathbb{S}^2,g_+)$
for which it holds that $Q_3(h,g_+)>0$ and $R_h>0$. It also tells us that case (1) in Theorem \ref{t:3d-rigidity-theorem} does not hold in the general setting of conformally compact Einstein spaces.
\end{example}

{\bf Acknowledgements.} The author is grateful to Alice Chang for many valuable discussions. The author 
should also thank Jeffrey Case for a lot of stimulating communications, especially for the discussions about Example \ref{e:hyperbolic} and the paper \cite{CaCh}.

\end{document}